\documentclass[reqno]{amsart}
\usepackage{amsmath,amsthm,xspace,enumerate,amscd,graphicx,epstopdf}
\usepackage{mathrsfs,float}
\usepackage{amssymb,url,color}
\usepackage{subfig}
\usepackage[pdftex]{hyperref}
\usepackage{array,multirow,makecell,booktabs,diagbox}
\usepackage[format=hang,font=small]{caption}

\newtheorem{lem}{Lemma}[section]

\newtheorem{thm}{Theorem}[section]
\newtheorem{cor}{Corollary}[section]

\newtheorem{rmk}{Remark}[section]

\numberwithin{equation}{section}

\def\esssup{\mathop{\rm ess\, sup}}

\title[MDP for mildly stationary autoregressive model]
{Moderate deviations for the mildly stationary autoregressive model with dependent errors}

\begin{document}

\maketitle

\author{
\begin{center}
Hui Jiang
\\
\vskip 0.2cm
{\it School of Mathematics, Nanjing University of Aeronautics and Astronautics
\\ Nanjing 210016, P.R.China\\
\vskip 0.1cm

huijiang@nuaa.edu.cn}

\vskip 0.4cm

Guangyu Yang 
\\
\vskip 0.2cm
{\it School of Mathematics and Statistics, Zhengzhou University
\\ Zhengzhou 450001, P.R.China\\
\vskip 0.1cm
guangyu@zzu.edu.cn}

\vskip 0.4cm
Mingming Yu
\vskip 0.2cm
{\it School of Mathematics, Nanjing University of Aeronautics and Astronautics
\\ Nanjing 210016, P.R.China\\
\vskip 0.1cm

mengyilianmeng@163.com}

\end{center}
}

\begin{abstract}

In this paper, we consider the normalized least squares estimator of the parameter in a mildly stationary first-order autoregressive (AR(1)) model with dependent errors which are modeled as a mildly stationary AR(1) process. By martingale methods, we establish the moderate deviations for the least squares estimators of the regressor and error, which can be applied to understand the near-integrated second order autoregressive processes. As an application, we also obtain the moderate deviations for the Durbin-Watson statistic.

\vskip10pt


\vskip5pt

\end{abstract}




\section{Introduction}

\noindent Regression asymptotics with roots at or near unity have played an important role in time series econometrics. In order to cover more general time series structure, it has become popular in econometric methodology to study the models which permit that the regressors and the errors have substantial heterogeneity and dependence over time. This work is devoted to analyse a dynamic first order autoregressive model where the errors are dependent. More precisely, we consider the asymptotic behavior of the least squares estimators of the following autoregressive model,
\begin{equation}\label{model}
\left\{
\begin{array}{ll}
X_{k,n}={\theta_n}X_{k-1,n}+\varepsilon_{k,n},\\
\varepsilon_{k,n}={\rho_n}\varepsilon_{k-1,n}+V_k,
\end{array}
\right.\quad k=1,2,\ldots,n,\;n\geq1,
\end{equation}
where the unknown parameters $\theta_n,\,\rho_n\in\mathbb{R}$, $(X_{k,n})_{0\leq k\leq n}$ is observed, and the noise $(V_k)_{k\geq1}$ is a sequence of independent and identically distributed (i.i.d.) random variables with zero mean and a finite variance $\sigma^2$. For convenience, let $X_{0,n}=\varepsilon_{0,n}=0$ for every $n\geq1$. We take two-stage methods to estimate the unknown parameters, $\theta_n$ and $\rho_n$. It is well-known that the least squares estimator of $\theta_n$ based on the observed variables is given by
\begin{equation}\label{theta}
\hat{\theta}_n=\frac{\sum_{k=1}^n{X_{k,n} X_{k-1,n}}}{\sum_{k=1}^n{X_{k-1,n}^2}},\quad n\geq1.
\end{equation}
To obtain the estimator of $\rho_n$, we substitute $\hat{\theta}_n$ for $\theta_n$ in (\ref{model}), and denote the residuals by
\begin{equation}\label{eq-1}
\hat\varepsilon_{k,n}=X_{k,n}-\hat{\theta}_nX_{k-1,n},\quad 1\leq k \leq n,
\end{equation}
then the least squares estimator of $\rho_n$ can be defined as
\begin{equation}\label{rho}
\hat{\rho}_n=\frac{\sum_{k=1}^n{\hat{\varepsilon}_{k,n}
\hat{\varepsilon}_{k-1,n}}}{\sum_{k=1}^n{\hat{\varepsilon}_{k-1,n}}^2},\quad n\geq1,
\end{equation}
where $\hat{\varepsilon}_{0,n}:=0$ for every $n\geq1$.
\vskip5pt

The model (\ref{model}) has a close connection with some existing models. Firstly, we fix the autoregressive coefficient $\theta_n$, i.e. let $\theta_n\equiv\theta$. If $\rho_n\equiv0$, then the model (\ref{model}) is precisely the classic autoregressive process with i.i.d. errors. In this case, the asymptotic behaviors of $\hat{\theta}_n$ have been examined thoroughly. For example, when the model is {\it stationary} ($|\theta|<1$), under some moment conditions, Anderson \cite{Anderson} showed asymptotic normality of $\hat{\theta}_n$. However, as pointed out previously by Anderson \cite{Anderson}, White \cite{White}, and Dickey \& Fuller \cite{DiFu1979}, the situation becomes more complicated for the {\it critical} case ($|\theta|=1$) and the {\it explosive} case ($|\theta|>1$), where the limiting distributions are functionals of Brownian motion and standard Cauchy, respectively. In addition, if the regressive coefficient $\rho_n$ in the errors is also fixed, i.e. $\rho_n\equiv\rho$, to answer some open problems on the Durbin-Watson statistic, Bercu and Pro\"{i}a \cite{Bercu-2013} investigated the asymptotic normality of the least squares estimators $\hat{\theta}_n$ and $\hat{\rho}_n$, in the stationary case, i.e. $|\theta|<1$ and $|\rho|<1$. Moreover, the continuous version of ~(\ref{model})
 is Ornstein-Uhlenbeck process~($\theta_n\equiv\theta,~\rho_n\equiv0$),
 or Ornstein-Uhlenbeck process driven by Ornstein-Uhlenbeck process ~($\theta_n\equiv\theta,~\rho_n\equiv\rho$)~(\cite{bercu-proia-savy}).
 For the associated statistical inferences, one can see~(\cite{chenY},\cite{gao-jiang-wang},\cite{gao-jiang-scm},
 \cite{hu-nualart},\cite{hu-nualart-zhou},\cite{jiang-liu-wang},\cite{shen-yin},\cite{shen-yin-correct}) and the references therein.
\vskip5pt

Secondly, we assume that the regression parameter is {\it time-varying}, i.e. $\theta_n$ depending on the sample size $n$. If $\rho_n\equiv0$ and $\theta_n=1+O(n^{-1})$, then the model (\ref{model}) turns to be the {\it near unit root processes} raised by Bobkoski \cite{Bobkoski1983} and Cavanagh \cite{Cavanagh1985} to understand the phenomenon that the discriminatory power of statistical tests for the presence of unit root is generally quite low against the alternative of root which is close, but not equal, to unity. Phillips \cite{Phillips} and Chan \& Wei \cite{Chan-Wei} established that the asymptotic distribution of $\hat{\theta}_n$ is the stochastic integration of some exponential function with respect to Brownian motion. If $\rho_n\equiv0$ and $\theta_n=1+\frac{\gamma}{k_n}$ for some fixed $\gamma\neq0$ where $k_n\to\infty$ and $k_n=o(n)$, then the model (\ref{model}) turns to be the {\it mildly autoregressive model} raised by Phillips \& Magdalinos \cite{Philllips-Magdalinos}. It is interesting that their results {\it match} the standard limit theory of the {\it time-invariant} model and partially bridge the stationary, the local to unity and the explosive cases.

\vskip5pt

Notice that, by simple calculations, the model (\ref{model}) can be written as
\begin{align*}
	X_{k,n}=(\theta_n+\rho_n)X_{k-1,n}-\theta_n\rho_nX_{k-2,n}+V_{k},
	\qquad k=1,2,\dots,n,\;n\geq1.
\end{align*}
Clearly, it is an AR(2) process and the roots of its characteristic polynomial both tend to unity as $\theta_n\to1$ and $\rho_n\to1$.
In fact, Nabeya \& Perron \cite{NabeyaPerron1994} once introduced the model (\ref{model}), where $\theta_n=1+\frac{\gamma_1}{n}$ and $\rho_n=1+\frac{\gamma_2}{n}$ for some fixed $\gamma_1$ and $\gamma_2$. They showed that the asymptotic distribution of $\hat{\theta}_n$ is some kind of functional of Brownian motion. And they also pointed out that the model (\ref{model}) can be regarded as an approximate version of the second order autoregressive process with two unit roots.

\vskip5pt

Motivated by above discussions, we will consider the {\it time-varying} model (\ref{model}) and devote to the asymptotic properties of the least squares estimators, $\hat{\theta}_n$ and $\hat{\rho}_n$. In the preprint \cite{JiangYuYang2015}, we proved the asymptotic normality of $\hat{\theta}_n$ and $\hat{\rho}_n$ when $|\theta_n|\to1$ and $|\rho_n|\to1$ both within stationary or explosive regions. The object of the present paper is to establish the moderate deviation principle of $\hat{\theta}_n$ and $\hat{\rho}_n$, when $|\theta_n|\to1$ and $|\rho_n|\to1$.

\vskip5pt

It is well known that the large deviation estimates have proved to be a crucial tool required to handle many questions in statistics, engineering, statistical mechanics, and applied probabilities.
For convenience, let us first recall some conceptions of large deviations.
Given a sequence of random variables $(X_n)_{n\geq1}$ with values in a topological space $(S,\mathcal{S})$, a sequence of positive numbers
$(a_n)_{n\geq1}$ such that $a_n\to\infty$, it is said that
$(X_n)_{n\geq1}$ satisfies the large deviation principle
with speed $a_n$ and rate function $I$, if $I$ is non-negative, lower semi-continuous and for any
Borel set $A$,
\begin{align}
-\inf_{x\in{A^0}}I(x)&\leq\liminf_{n\to\infty}\frac{1}{a_n}\log
{P}(X_n\in{A})\nonumber\\
&\leq\limsup_{n\to\infty}\frac{1}{a_n}\log
{P}(X_n\in{A})\leq
-\inf_{x\in{\bar{A}}}I(x),
\end{align}
where $A^\circ$ and $\bar{A}$ denote the interior and closure of $A$ respectively; we say the rate function $I$ is {\it good} if its level set $\{x\in S: I(x)\leq\alpha\}$ is compact for all $\alpha>0$.
If additionally, we assume that $X_n$ satisfies a fluctuation theorem such as central limit theorem, that is,
there exists a sequence of positive numbers $b_n\to\infty$ such that $b_n(X_n-c_n)\stackrel{\mathcal{L}}\rightarrow Y$,
where $c_n$ is a sequence of real numbers, $Y$ has a non-degenerate distribution, and $\stackrel{\mathcal{L}}\rightarrow$ means the convergence in distribution,  then, in general, a large deviation principle for $(Y_n:=r_n(X_n-c_n))_{n\geq 1}$ is also called a moderate deviation principle for $(X_n)_{n\geq 1}$, where $r_n$ is an intermediate
scale between $1$ and $b_n$, that is, $r_n\to \infty $ and $b_n/r_n\to\infty$.
For a full description of large deviation theory and its applications, we refer to Dembo \& Zeitouni \cite{Dembo}.
\vskip5pt

For the classic autoregressive model, i.e. $\theta_n\equiv\theta$ and $\rho_n\equiv0$, Bercu \cite{Bercu-2001} and Worms \cite{Worms-1} provided the large deviation estimates of $\hat{\theta}_n$ in the stationary, critical and explosive cases when the noise is Gaussian. Miao \& Shen \cite{Miao-Shen} proved the moderate deviation principle of $\hat{\theta}_n$ for general i.i.d. noise which satisfies Gaussian integrability. Recently, Miao {\it et al.} \cite{Miao-Yang} extended the results in \cite{Miao-Shen} to the time-varying model, and obtained the moderate deviations of $\hat{\theta}_n$ when $\rho_n\equiv0$ and $\theta_n\to1$ within the stationary regions, while Pro\"{i}a \cite{Proia2020} extended the results in \cite{Miao-Yang} to the autoregressive process of any order; interestingly, their results both match the standard limit theory of the time-invariant model. Bitseki Penda {\it et al.} \cite{Penda} studied the moderate deviations of $\hat{\theta}_n$ and $\hat{\rho}_n$ for the model (\ref{model}) when $\theta_n\equiv\theta$ and $\rho_n\equiv\rho$ both in stationary cases, and the noise satisfies a less restrictive Chen-Ledoux type condition. As an application, they obtained the moderate deviation principle of the Durbin-Watson statistic.

\vskip5pt

The main contribution of this paper is to extend the results in \cite{Penda} to the time-varying model (\ref{model}) when $|\theta_n|\to1$ and $|\rho_n|\to1$ both within the stationary regions. It is interesting that, when $\theta_n$ and $\rho_n$ have opposite signs, the estimators $\hat{\theta}_n$ and $\hat{\rho}_n$ have the same rates of convergence and rate functions.
In addition, we remark that the methods of proof mainly rely on the deviation inequalities for martingale arrays and our results can also be applied to understand the near-integrated second order autoregressive processes. The rest of this article is organized as follows. The next section is devoted to the descriptions of our main results.
 Then the proofs of main results are completed in Section 3. Finally, the proofs of some technique lemmas are given in Section 4.



\section{Main results}\label{sec2}


\subsection{Assumptions}

For the model (\ref{model}), let $(V_n)_{n\geq1}$ be a sequence of real valued i.i.d. random variables with zero mean and a finite variance $\sigma^2$,
and~$V_1$ satisfies the Gaussian integrability condition, i.e. for some $t_0>0$,
\begin{equation}\label{exp-moment}
E\left(\exp\{t_0V_1^2\}\right)<\infty.
\end{equation}

Consider the following two cases of ~$\theta_n$ and $\rho_n$:
\begin{align*}
&{\rm{(Case~I)}}\qquad\qquad\;\;\theta_n=1+\frac{\gamma_1}{\kappa_n},\quad \rho_n=1+\frac{\gamma_2}{\kappa_n},\qquad \gamma_1<0,\;\gamma_2<0;
\\
&{\rm{(Case~II)}}\qquad\qquad\theta_n=1+\frac{\gamma_1}{\kappa_n},\quad \rho_n=-1-\frac{\gamma_2}{\kappa_n},\qquad \gamma_1<0,\;\gamma_2<0.
\end{align*}
For the above cases, $(b_n)_{n\geq1}$ and $(\lambda_n)_{n\geq1}$ are two sequences of positive numbers and satisfy that
\vskip3pt

\noindent{(H-I)}\quad $(b_n)_{n\geq1}$ is a sequence of increasing positive numbers satisfying
$$
b_n\to\infty,\quad\frac{n}{b_n^6\kappa_n^2}\to\infty,\quad \frac{n}{b_n^2\kappa_n^{5}}\to\infty, ~\text{and}~
\quad\frac{nb_n^2}{\kappa_n^5\log^2n}\to\infty;
$$

\noindent{(H-II)}\quad $(\lambda_n)_{n\geq1}$ is a sequence of increasing positive numbers satisfying
$$
\lambda_n\to\infty,\quad\frac{n}{\lambda_n^6\kappa_n^6}\to\infty,\quad \frac{n}{\lambda_n^2\kappa_n^{11}}\to\infty, ~\text{and}~
\quad\frac{n\lambda_n^2}{\kappa_n^7\log^2n}\to\infty.
$$

Moreover, we need the following Chen-Ledoux condition (\cite{Penda},\cite{Chen},\cite{Eichelsbacher},\cite{Ledoux}):
\vskip3pt

\noindent(C-L~($a_n$))\quad as $n\to\infty$,
\begin{equation}\label{C-L}
\frac{1}{a_n^2}\log{nP\left(|V_1|^4>a_n\sqrt{n}\right)}\to-\infty,
\end{equation}
where~$a_n=b_n$ or $\lambda_n$. Note that, if $V_1$ is a Gaussian random variable, then the condition~(C-L~($a_n$)) holds.
\vskip5pt

For the convenience of statement, throughout this paper, we always assume that {\it under (Case~I), conditions (H-I) and (C-L~($b_n$)) hold;
under (Case~II), conditions (H-II) and (C-L~($\lambda_n$)) are valid.}


\subsection{Moderate deviations}

For some convenience, denote
\begin{align}\label{thetastar}
\theta^*_n=\frac{\theta_n+\rho_n}{1+\theta_n\rho_n},
\quad\rho^*_n=\theta_n\rho_n\theta^*_n,
\quad d^*_n=2(1-\rho^*_n)
\end{align}
and
\begin{align}\label{Gamma}
\Gamma=
\begin{pmatrix}
-\frac{\gamma_1\gamma_2(\gamma_1+\gamma_2)}{2} & 0
\\
0&-2(\gamma_1+\gamma_2)
\end{pmatrix}.
\end{align}

\begin{thm}\label{MDP-theta-rho}
Under (Case~I),
$$
\left\{\frac{1}{b_n}\begin{pmatrix}
\sqrt{n\kappa_n^3}(\hat{\theta}_n-\theta^*_n)\\\sqrt{n\kappa_n}(\hat{\rho}_n-\rho^*_n)
\end{pmatrix},\,
n\geq1\right\}
$$
satisfies the large deviation principle with speed $b_n^2$ and good rate function
\begin{align}
I_{\theta,\rho}(x):=\frac{x^{\tau}\Gamma^{-1} x}{2},\quad x\in\mathbb{R}^2.
\end{align}

In particular, $\left\{\frac{\sqrt{n\kappa_n^3}}{b_n}(\hat{\theta}_n-\theta^*_n),\, n\geq1\right\}$ and $\left\{\frac{\sqrt{n\kappa_n}}{b_n}(\hat{\rho}_n-\rho^*_n),\,n\geq1\right\}$ satisfy the large deviation principle with speed $b_n^2$ and good rate functions
\begin{align}
I_{\theta}(x):=-\frac{x^2}{\gamma_1\gamma_2(\gamma_1+\gamma_2)}\quad {\rm and}\quad
I_{\rho}(x):=-\frac{x^2}{4(\gamma_1+\gamma_2)}, \quad x\in\mathbb{R},
\end{align}
respectively.
\end{thm}


From the preprint paper of Jiang {\it et al.} \cite{JiangYuYang2015}, we know that, under (Case~II), the covariance of limiting distribution of
$\hat{\theta}_n$ and $\hat{\rho}_n$ is singular. Therefore, in this situation, we study the moderate deviations for each estimator individually.

\begin{thm}\label{MDP-theta-rho-II}
Under (Case~II),
$$
\left\{\frac{\sqrt{n/\kappa_n}}{\lambda_n}(\hat{\theta}_n-\theta^*_n),\, n\geq1\right\},\quad
\left\{\frac{\sqrt{n/\kappa_n}}{\lambda_n}(\hat{\rho}_n-\rho^*_n),\,n\geq1\right\}
$$
satisfy the large deviation principle with speed $\lambda_n^2$ and same good rate functions
\begin{align}
J(x):=-\frac{(\gamma_1+\gamma_2)^3x^2}{16\gamma_1\gamma_2}, \quad x\in\mathbb{R}.
\end{align}
\end{thm}

\begin{rmk}
For the time-invariant models, i.e. $\theta_n\equiv\theta$ and $\rho_n\equiv\rho$, Bitseki Penda et al. \cite{Penda} showed that, when $\theta=-\rho$, $\left\{\frac{\sqrt{n}}{\lambda_n}(\hat{\theta}_n-\theta^*_n), n\geq1\right\}$ and $\left\{\frac{\sqrt{n}}{\lambda_n}(\hat{\rho}_n-\rho^*_n), n\geq1\right\}$ have different rate functions.
\end{rmk}

As pointed out by King \cite{King0}, the following Durbin-Watson statistic,
\begin{equation}\label{D}
\hat{d}_n=\frac{\sum_{k=1}^n(\hat\varepsilon_{k,n}-\hat\varepsilon_{k-1,n})^2}
{\sum_{k=1}^n\hat\varepsilon_{k,n}^2},\quad n\geq1,
\end{equation}
plays an important role in the test of the serial correlation.
As an application of Theorem \ref{MDP-theta-rho} and Theorem \ref{MDP-theta-rho-II}, we have the following result.

\begin{cor}\label{mdp-DW}
Define
\begin{align}
I_{d}(x)=-\frac{x^2}{16(\gamma_1+\gamma_2)},\quad
J_{d}(x)=-\frac{(\gamma_1+\gamma_2)^3x^2}{64\gamma_1\gamma_2},\quad x\in\mathbb{R},
\end{align}
then, for any $x>0$, we have
\begin{align}
\left\{\begin{array}{ll}
\lim_{n\to\infty}\frac{1}{b_n^2}\log P\left(\frac{\sqrt{n\kappa_n}}{b_n}
|\hat{d}_n-d_n^*|\geq x\right)=-I_d(x),& \rm{Case~I}\\
\lim_{n\to\infty}\frac{1}{\lambda_n^2}\log P\left(\frac{\sqrt{n/\kappa_n}}{\lambda_n}
|\hat{d}_n-d_n^*|\geq x\right)=-J_d(x),& \rm{Case~II}
\end{array}\right..
\end{align}
\end{cor}


\subsection{Discussions}

It is worthwhile to give some additional comments on our results and other related problems.

\begin{enumerate}
  \item 
  In fact, by taking the same line of the proofs of Theorems \ref{MDP-theta-rho} and \ref{MDP-theta-rho-II}, there are similar moderate deviations results for the cases, $\theta_n\to-1,\rho_n\to-1$ and $\theta_n\to-1,\rho_n\to1$, both within the stationary regions. For the purpose of simplicity, in the present paper we only state the results for the case, $\theta_n\to1,\rho_n\to1$ and $\theta_n\to1,\rho_n\to-1$, both within the stationary regions, i.e. (Case I) and (Case II). Fortunately, if the distribution of $V_1$ is symmetric, following the simple procedure below, we can show that Theorems \ref{MDP-theta-rho} and Theorem \ref{MDP-theta-rho-II} hold for the cases $\theta_n\to-1,\rho_n\to-1$ and $\theta_n\to-1,\rho_n\to1$ both within the stationary regions, respectively.
  Suppose that
  \begin{align*}
  \left\{
  \begin{array}{ll}
  X_{k,n}&={\theta_n}X_{k-1,n}+\varepsilon_{k,n},\\
  \varepsilon_{k,n}&={\rho_n}\varepsilon_{k-1,n}+V_k,
  \end{array}
  \right.\quad k=1,2,\ldots,n,\;n\geq1,
  \end{align*}
  where the unknown parameters $\theta_n,\rho_n\in(-1,0)$, and $(V_k)_{k\geq1}$ is a sequence of i.i.d. random variables with a symmetric distribution. Denote
  \begin{align*}
  &\alpha_n=-\theta_n,\quad \beta_n=-\rho_n,\\
  Y_{k,n}=(-1)^{k}X_{k,n},&\quad \eta_{k,n}=(-1)^k\varepsilon_{k,n},\quad W_k=(-1)^kV_k,
  \end{align*}
  then $\alpha_n,\beta_n\in(0,1)$, $(W_k)_{k\geq1}$ is a sequence of i.i.d. random variables with the same common distribution as that of $V_1$ and
  \begin{align*}
  \left\{
  \begin{array}{ll}
  Y_{k,n}&={\alpha_n}Y_{k-1,n}+\eta_{k,n},\\
  \eta_{k,n}&={\beta_n}\eta_{k-1,n}+W_k,
  \end{array}
  \right.\quad k=1,2,\ldots,n,\;n\geq1.
  \end{align*}
  Putting
  \begin{align*}
  \theta_n^*=\frac{\theta_n+\rho_n}{1+\theta_n\rho_n},
  \quad&\rho_n^*=\theta_n\rho_n\theta_n^*,\quad
  \alpha_n^*=\frac{\alpha_n+\beta_n}{1+\alpha_n\beta_n},
  \quad\beta_n^*=\alpha_n\beta_n\alpha_n^*,\\
  \hat{\theta}_n=&\frac{\sum_{k=1}^n{X_{k,n} X_{k-1,n}}}{\sum_{k=1}^n{X_{k-1,n}^2}},\quad \hat{\alpha}_n=\frac{\sum_{k=1}^n{Y_{k,n} Y_{k-1,n}}}{\sum_{k=1}^n{Y_{k-1,n}^2}},\\
  \hat{\rho}_n=&\frac{\sum_{k=1}^n{\hat{\varepsilon}_{k,n}
\hat{\varepsilon}_{k-1,n}}}{\sum_{k=1}^n{\hat{\varepsilon}_{k-1,n}}^2},\quad
  \hat{\beta}_n=\frac{\sum_{k=1}^n{\hat{\eta}_{k,n}
\hat{\eta}_{k-1,n}}}{\sum_{k=1}^n{\hat{\eta}_{k-1,n}}^2},
  \end{align*}
  where $\hat{\eta}_{k,n}=Y_{k,n}-\hat{\alpha}_nY_{k-1,n}$, we know that
  \[
  \theta_n^*=-\alpha_n^*, \quad \rho_n^*=-\beta_n^*,
  \]
  and
  \[
  \hat{\alpha}_n=-\hat{\theta}_n, \quad \hat{\beta}_n=-\hat{\rho}_n.
  \]
  Given $\kappa_n=o(n)$ and $\kappa_n\to\infty$, if let $\theta_n=-1+\frac{\gamma_1}{\kappa_n},\rho_n=-1+\frac{\gamma_2}{\kappa_n}$ for some $\gamma_1>0,\gamma_2>0$ and assume the conditions (H-I) and (C-L($b_n$)) are satisfied, then Theorem \ref{MDP-theta-rho} holds for the least squares estimators, $\hat{\alpha}_n$ and $\hat{\beta}_n$, hence also holds for $\hat{\theta}_n$ and $\hat{\rho}_n$. By the same procedure we can show that Theorem \ref{MDP-theta-rho-II} holds under the case, $\theta_n\to-1,\rho_n\to1$, both within the stationary regions, if the distribution of $V_1$ is symmetric and the conditions (H-II) and (C-L($\lambda_n$)) are satisfied.
  \vskip5pt

    \item Theorem \ref{MDP-theta-rho}, Theorem \ref{MDP-theta-rho-II} and Corollary \ref{mdp-DW} {\it match} the following standard stationary moderate deviations for the {\it time-invariant} model (i.e. the parameters $\theta_n,\rho_n$ in the model (\ref{model}) do not depend on $n$) developed in \cite{Penda},
       \begin{equation}\label{Penda-old}
       \aligned
       \left\{\begin{array}{lll}
       \frac{1}{b_n^2}\log P\left(\frac{\sqrt{n}}{b_n}
       |\hat{\theta}_n-\theta^*|\geq x\right)\thicksim-\frac{x^2}{2\sigma_\theta^2},\\
       \frac{1}{b_n^2}\log P\left(\frac{\sqrt{n}}{b_n}
       |\hat{\rho}_n-\rho^*|\geq x\right)\thicksim-\frac{x^2}{2\sigma_\rho^2},\\
       \frac{1}{b_n^2}\log P\left(\frac{\sqrt{n}}{b_n}
       |\hat{d}_n-d^*|\geq x\right)\thicksim-\frac{x^2}{2\sigma_d^2},
       \end{array}\right.
       \endaligned
       \end{equation}
  where, $\sigma_d^2=4\sigma_\rho^2$ and
  \begin{equation}\label{Penda-old-1}
  \aligned
  \left\{\begin{array}{ll}
  \sigma_\theta^2&=\frac{(1-\theta^2)(1-\theta\rho)(1-\rho^2)}{(1+\theta\rho)^3},\\
  \sigma_\rho^2&=\frac{(1-\theta\rho)
  ((\theta+\rho)^2(1+\theta\rho)^2+\theta^2\rho^2(1-\theta^2)(1-\rho^2))}
  {(1+\theta\rho)^3}.
  \end{array}\right.
  \endaligned
  \end{equation}
 Notice that we omit the subscripts $n$ of $\theta^*,\rho^*$ and $d^*$ because in this case they do not depend on $n$. Indeed, substituting $1-\theta_n^2=-\frac{2\gamma_1}{\kappa_n}(1+o(1))$,  $1-\rho_n^2=-\frac{2\gamma_2}{\kappa_n}(1+o(1))$, and $\theta_{n}^*,\rho_n^*,d_n^*$ into the above equations (\ref{Penda-old-1}) and (\ref{Penda-old}), we can get the asymptotic approximation
  $$
  \left\{\begin{array}{ll}
\sigma_\theta^2\thicksim-\frac{\gamma_1\gamma_2(\gamma_1+\gamma_2)}{2\kappa_n^3},
\quad\sigma_\rho^2\thicksim-\frac{2(\gamma_1+\gamma_2)}{\kappa_n},
\quad \sigma_d^2\thicksim-\frac{8(\gamma_1+\gamma_2)}{\kappa_n},& \textrm{Case~I}\\
\sigma_\theta^2\thicksim-\frac{8\kappa_n\gamma_1\gamma_2}{(\gamma_1+\gamma_2)^3},
\quad\sigma_\rho^2\thicksim-\frac{8\kappa_n\gamma_1\gamma_2}{(\gamma_1+\gamma_2)^3},
\quad\sigma_d^2\thicksim-\frac{32\kappa_n\gamma_1\gamma_2}{(\gamma_1+\gamma_2)^3},& \textrm{Case~II}\\
\end{array}\right.,
  $$
which just correspond to our results.
  \vskip5pt

  \item
  Since $|\theta_n|\to1$ and $|\rho_n|\to1$, are both within the stationary regions, our results may provide a bridge between those for local to unity processes and those that apply under stationarity. Given $\gamma_1,\gamma_2\in(-1,0)$, let $\kappa_n=n^\delta$ for some $\delta\in(0,1)$,
      then, under (Case I), by Theorem \ref{MDP-theta-rho},
      Corollary \ref{mdp-DW}, and letting $\delta\to0$, we have
  \begin{equation}\label{simulation1}
  \aligned
  \left\{\begin{array}{lll}
  \frac{1}{b_n^2}\log P\left(\frac{\sqrt{n}}{b_n}
  |\hat{\theta}_n-\theta_n^*|\geq x\right)\thicksim\frac{x^2}{\gamma_1\gamma_2(\gamma_1+\gamma_2)},\\
  \frac{1}{b_n^2}\log P\left(\frac{\sqrt{n}}{b_n}
  |\hat{\rho}_n-\rho_n^*|\geq x\right)\thicksim\frac{x^2}{4(\gamma_1+\gamma_2)},\\
  \frac{1}{b_n^2}\log P\left(\frac{\sqrt{n}}{b_n}
  |\hat{d}_n-d_n^*|\geq x\right)\thicksim\frac{x^2}{16(\gamma_1+\gamma_2)},
  \end{array}\right.
  \endaligned
  \end{equation}
  however, the correct results for the {\it time-invariant} model when $\theta=1+\gamma_1$ and $\rho=1+\gamma_2$ (stationary case), should be
  \begin{equation}\label{simulation2}
  \aligned
  \left\{\begin{array}{lll}
  \frac{1}{b_n^2}\log P\left(\frac{\sqrt{n}}{b_n}
  |\hat{\theta}_n-\theta^*|\geq x\right)\thicksim-\frac{x^2}{2\sigma_{\theta}^2},\\
  \frac{1}{b_n^2}\log P\left(\frac{\sqrt{n}}{b_n}
  |\hat{\rho}_n-\rho^*|\geq x\right)\thicksim-\frac{x^2}{2\sigma_{\rho}^2},\\
  \frac{1}{b_n^2}\log P\left(\frac{\sqrt{n}}{b_n}
  |\hat{d}_n-d^*|\geq x\right)\thicksim-\frac{x^2}{8\sigma_{\rho}^2},
  \end{array}\right.
  \endaligned
  \end{equation}
  where $\sigma_\theta^2, \sigma_\rho^2$ and $\sigma_d^2$ are defined as in (\ref{Penda-old-1}). That is to say, the {\it time-varying} coefficients in the model (\ref{model}) change the variances of $\hat{\theta}_n$ and $\hat{\rho}_n$. One can refer to  \cite{Penda} for more details. Similar phenomena also appear under (Case II). In particular, if $\theta_n=-\rho_n$, i.e. $\gamma_1=\gamma_2=:\gamma$, by the same procedures we can get
  \begin{equation}
  \aligned
  \left\{\begin{array}{ll}
  \frac{1}{\lambda_n^2}\log P\left(\frac{\sqrt{n}}{\lambda_n}
  |\hat{\theta}_n-\theta_n^*|\geq x\right)\thicksim\frac{\gamma x^2}{2},
  \\
  \frac{1}{\lambda_n^2}\log P\left(\frac{\sqrt{n}}{\lambda_n}
  |\hat{\rho}_n-\rho_n^*|\geq x\right)\thicksim\frac{\gamma x^2}{2}.
  \end{array}\right.
  \endaligned
  \end{equation}
  However, here the correct variances for the {\it time-invariant} model should be
  \begin{align}
  \sigma^2_{\theta}=-\frac{\gamma^2+2\gamma+2}{\gamma^2+2\gamma},\qquad
  \sigma^2_{\rho}=-\frac{(1+\gamma)^4(\gamma^2+2\gamma+2)}{\gamma^2+2\gamma}.
  \end{align}

  To illustrate above findings, we do a small simulation study on moderate deviation principles for $\hat{\theta}_n$ and $\hat{\rho}_n$ under (Case I), i.e. Theorem \ref{MDP-theta-rho}. Assume the noise $(V_k)_{k\geq1}$ is a sequence of i.i.d. standard normal random variables or uniform random variables on $(-1,1)$. Given $\gamma_1=-0.9, \gamma_2=-0.8$, let $b_n=(\log n)^3$ and $\kappa_n=n^\delta$ with $\delta=0.01$. We take the sample size $n=100, 200$ and repeat $1000$ times. The {\it red} and {\it green curves} in the figures below denote the parabolic curves on the right-side of equations (\ref{simulation1}) and (\ref{simulation2}) respectively, and the {\it blue points} denote the simulation curves of the following functions,
  \begin{align*}
  y_\theta(x):=\frac{1}{b_n^2}\log P\Big(\frac{\sqrt{n\kappa_n^3}}{b_n}
  |\hat{\theta}_n-\theta_n^*|\geq x\Big)
  \end{align*}
  and
  \begin{align*}
  y_\rho(x):=\frac{1}{b_n^2}\log P\Big(\frac{\sqrt{n\kappa_n}}{b_n}
  |\hat{\rho}_n-\rho_n^*|\geq x\Big),
  \end{align*}
  respectively. From the simulations, we can see that, the time-varying coefficients in the model (\ref{model}) underestimate the variance of $\hat{\theta}_n$ but overestimate the variance of $\hat{\rho}_n$. Finally, we remark that the simulation results about $\hat{\theta}_n$ are very good, while the simulation results about $\hat{\rho}_n$ are not so satisfactory which may be caused by the plug-in method when estimating $\rho_n$, i.e. (\ref{rho}).

\begin{figure}[H]
\begin{center}
\includegraphics[width=5cm,height=5cm]{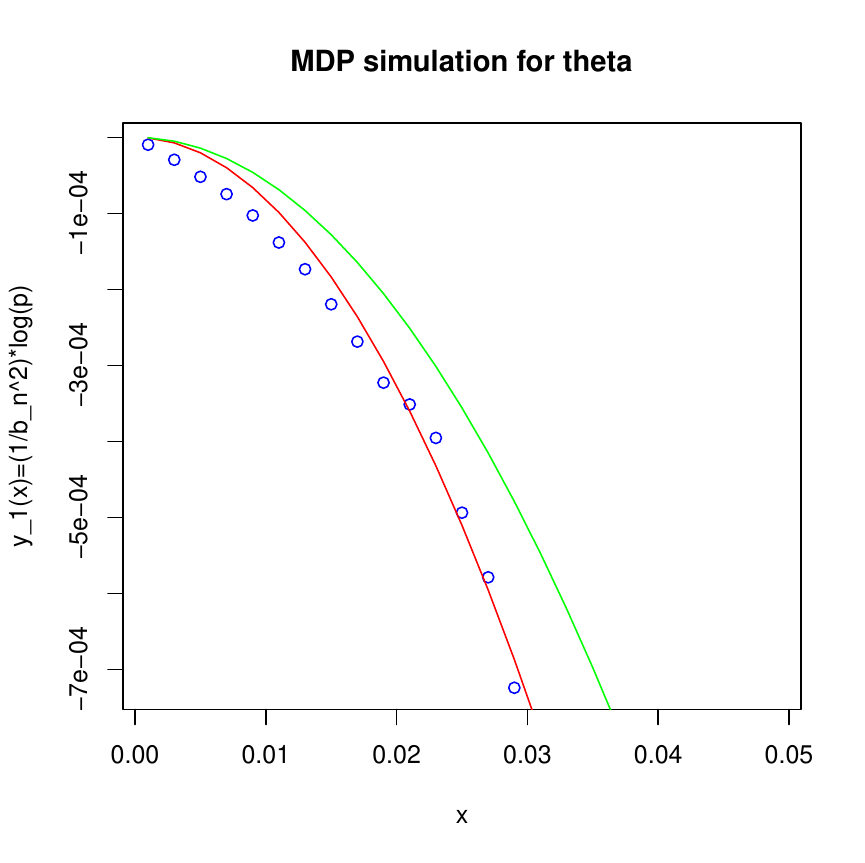}\;\;\includegraphics[width=5cm,height=5cm]{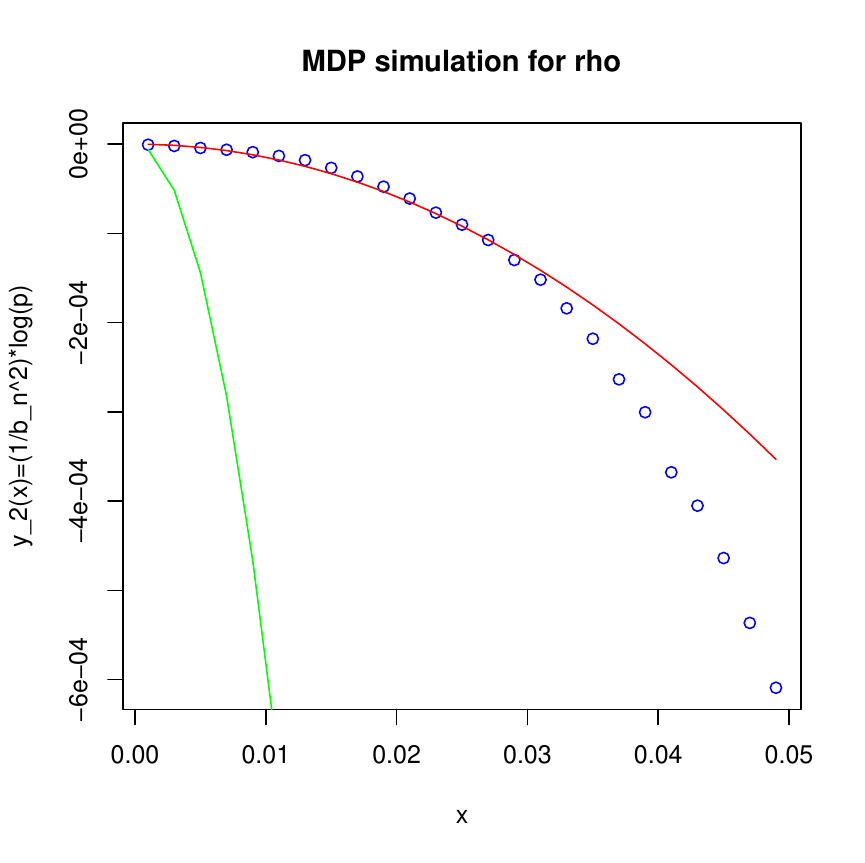}
\end{center}
\caption{\small{Sample size $n=100$ and $V_1\sim\mathscr{N}(0,1)$.}}
\end{figure}

\begin{figure}[H]
\begin{center}
\includegraphics[width=5cm,height=5cm]{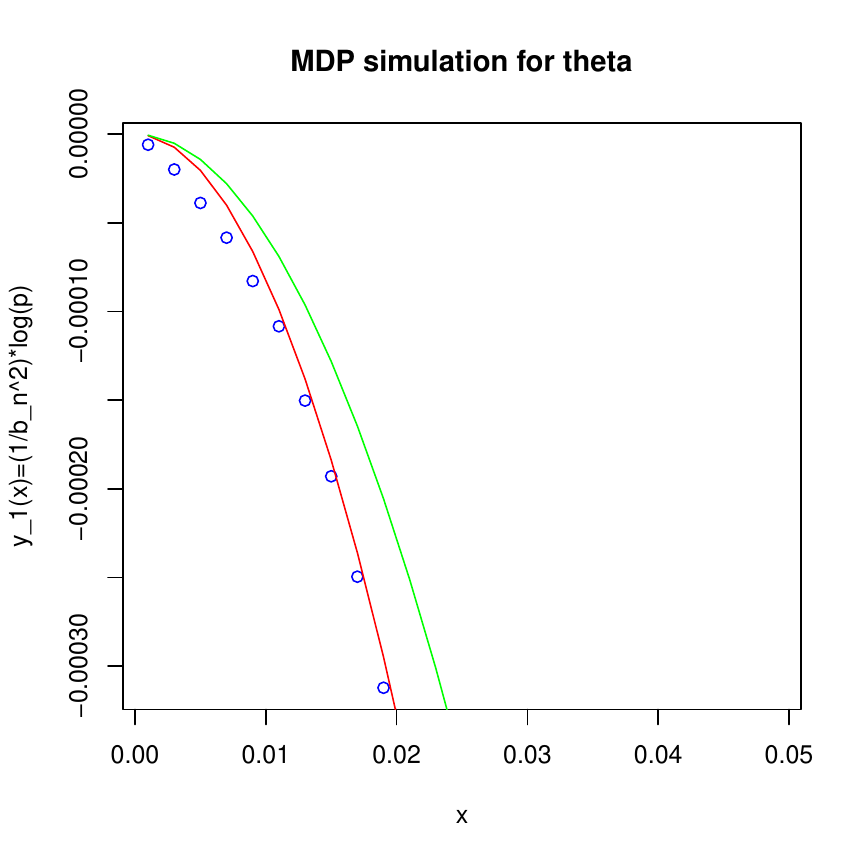}\;\;\includegraphics[width=5cm,height=5cm]{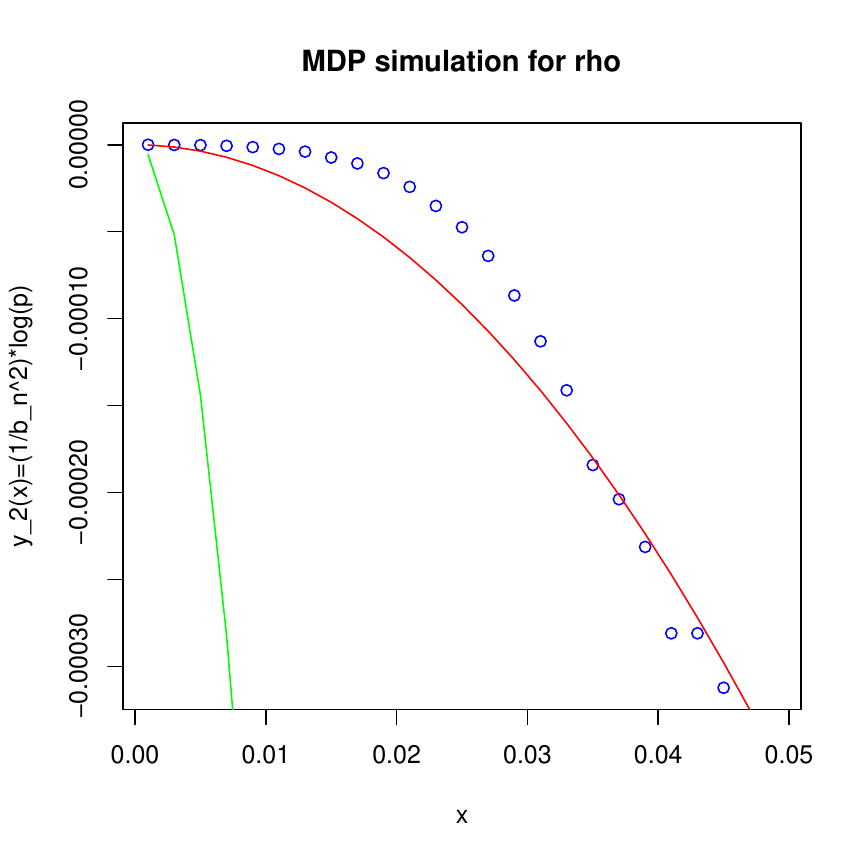}
\end{center}
\caption{\small{Sample size $n=200$ and $V_1\sim\mathscr{N}(0,1)$.}}
\end{figure}

\begin{figure}[H]
\begin{center}
\includegraphics[width=5cm,height=5cm]{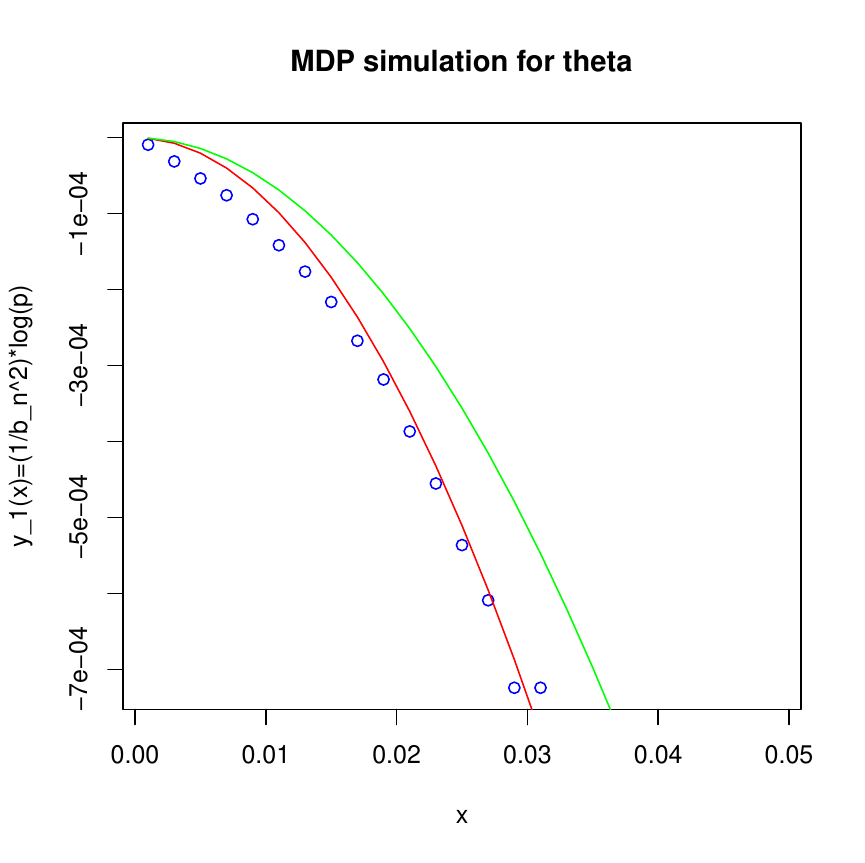}\;\;\includegraphics[width=5cm,height=5cm]{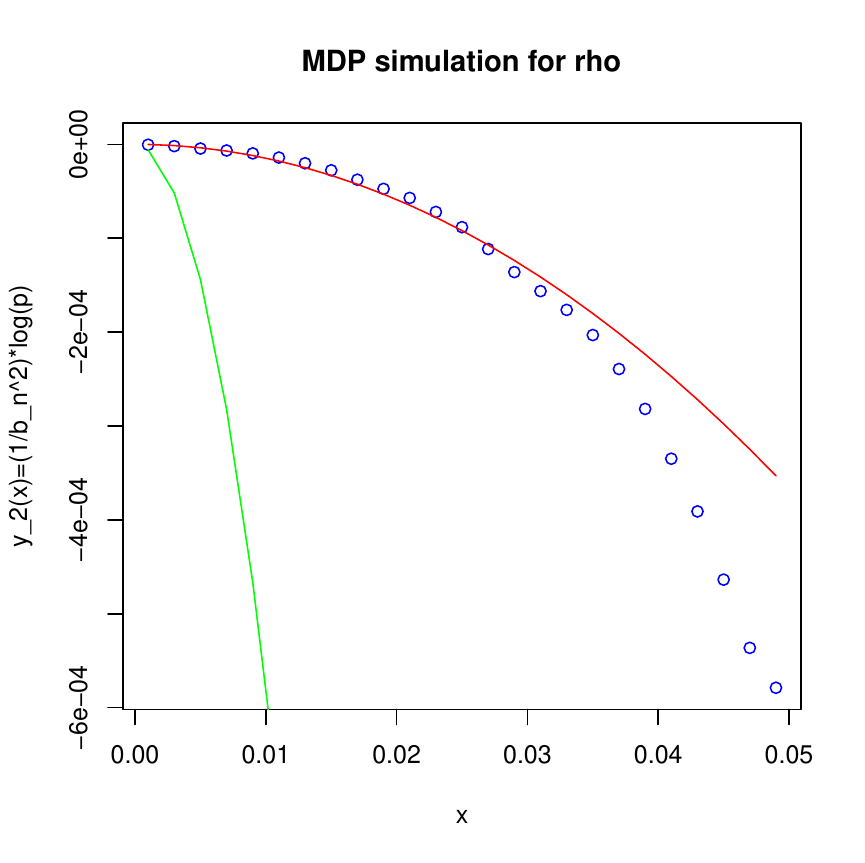}
\end{center}
\caption{\small{Sample size $n=100$ and $V_1\sim U(-1,1)$.}}
\end{figure}

\begin{figure}[H]
\begin{center}
\includegraphics[width=5cm,height=5cm]{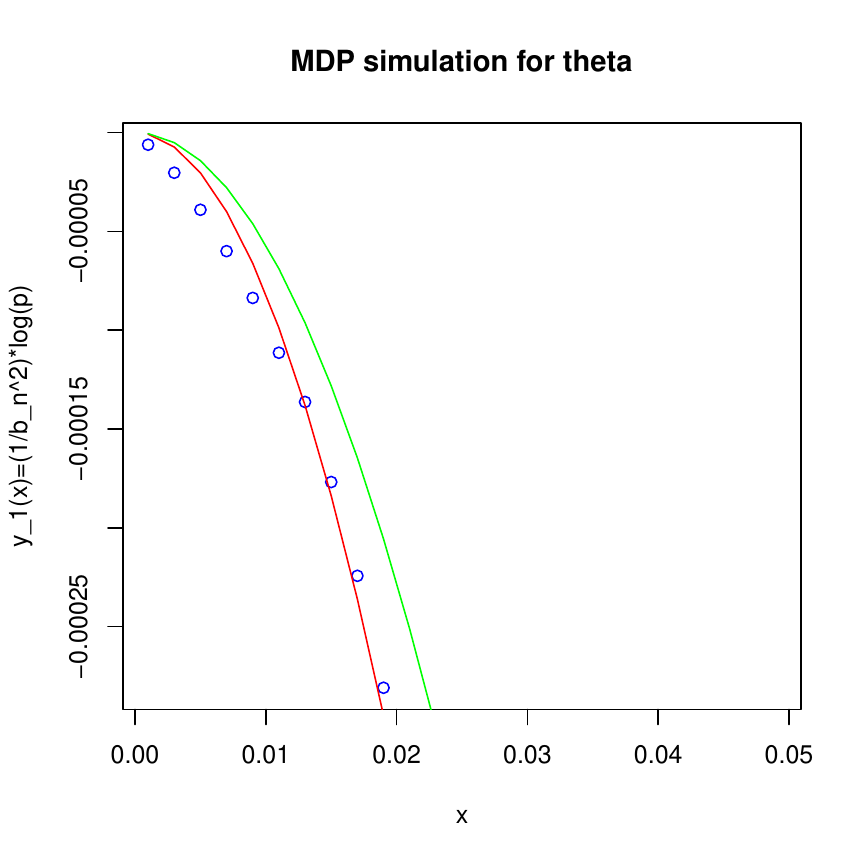}\;\;\includegraphics[width=5cm,height=5cm]{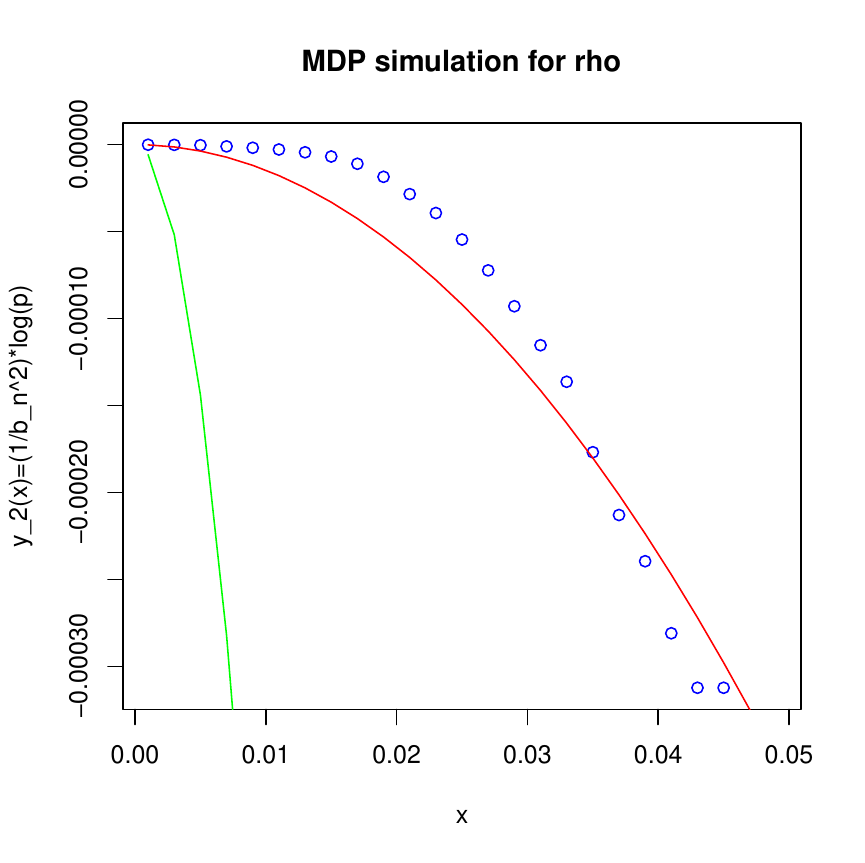}
\end{center}
\caption{\small{Sample size $n=200$ and $V_1\sim U(-1,1)$.}}
\end{figure}

\end{enumerate}



\section{Moderate deviations for $\hat{\theta}_n, \hat{\rho}_n$ and $\hat{d}_n$}\label{sec4}

This section is devoted to the proofs of moderate deviations for the estimators $\hat{\theta}_n, \hat{\rho}_n$ and $\hat{d}_n$ under
(Case~I) and (Case~II).

\subsection{Decomposition of estimators}

Based on the ideas in Bercu \& Pro\"{i}a \cite{Bercu-2013} and Bitseki Penda {\it et al.} \cite{Penda}, in order to prove our main results,
 we first deal with the decompositions of $\hat{\theta}_n-\theta_n^*$ and $\hat{\rho}_n-\rho_n^*$.

Now, for all $1\leq l\leq n$, let
\begin{align}\label{martingales}
M_{l,n}=\sum_{k=1}^l{X_{k-1,n}V_k},
\quad N_{l,n}=\sum_{k=2}^l{X_{k-2,n}V_k},\quad U_{l,n}=\sum_{k=1}^l{\varepsilon_{k-1,n}V_k},
\end{align}
and
\begin{equation}\label{variance-covariance}
P_{l,n}=\sum_{k=1}^l{X_{k,n}X_{k-1,n}},\; Q_{l,n}=\sum_{k=1}^lX_{k,n}\varepsilon_{k,n},\; S_{l,n}=\sum_{k=1}^lX_{k,n}^2,\;T_{l,n}=\sum_{k=1}^l\varepsilon_{k,n}^2.
\end{equation}
In addition, denote $M_n:=M_{n,n}$, and the same definitions for $N_n$, $U_n$, $P_n$, $Q_n$, $S_n$ and $T_n$. Then, from (\ref{model}), (\ref{theta}), and (\ref{thetastar}), it follows that
\begin{equation}\label{theta-decom}
\hat{\theta}_n-\theta^*_n=\frac{1}{1+\theta_n\rho_n}\cdot\frac{M_n}{S_{n-1,n}}
+\frac{\theta_n\rho_n}{1+\theta_n\rho_n}\cdot\frac{X_{n,n}X_{n-1,n}}{S_{n-1,n}}.
\end{equation}
As for the decomposition of $\hat{\rho}_n-\rho_n^*$, we need more notations. For $1\leq l\leq n$, let
\begin{align}\label{notation4}
I_{l,n}=\sum_{k=1}^l{\hat{\varepsilon}_{k,n}\hat{\varepsilon}_{k-1,n}},\quad J_{l,n}=\sum_{k=1}^l{\hat\varepsilon_{k,n}^2},
\quad I_{n,n}:=I_n,\quad J_{n,n}:=J_n.
\end{align}
By simple calculations, we have
\begin{align}\label{fenjierho}
&J_{n-1,n}(\hat{\rho}_n-\rho^*_n)\nonumber\\
&=\left(\frac{1+\theta^*_n\rho^*_n}{1+\theta_n\rho_n}-\frac{\theta_n^*}{\theta_n}\right)M_n
+\frac{\theta^*_n}{\theta_n}U_n-(\hat{\theta}_n-\theta^*_n)H_n+\xi_n^I-\rho^*_n\xi_n^J,
\end{align}
where
\begin{equation}\label{xiIJPQ}
\aligned
&H_n=F_n-\rho^*_nG_n,\quad F_n=S_n+W_n-(\hat{\theta}_n+\theta^*_n)P_n,
\quad G_n=2P_n-(\hat{\theta}_n+\theta^*_n)S_n,\\
&W_n=\sum_{k=2}^n{X_{k,n}X_{k-2,n}},
\quad\xi_n^I=\hat{\theta}_nX_{n,n}^2-\hat{\theta}_n^2X_{n,n}X_{n-1,n}+\frac{1+(\theta^*_n)^2}{1+\theta_n\rho_n}\xi_n^P-\theta^*_n\xi_n^Q,\\
&\xi_n^J=-X_{n,n}^2+2\hat{\theta}_nX_{n,n}X_{n-1,n}-\hat{\theta}_n^2(X_{n,n}^2+X_{n-1,n}^2)-\frac{2\theta^*_n}{1+\theta_n\rho_n}\xi_n^P,\\
&\xi_n^P=\theta_n\rho_nX_{n}X_{n-1,n}-(\theta_n+\rho_n)X_{n,n}^2,\\
&\xi_n^Q=\theta^*_n\xi_n^P-(\theta_n+\rho_n)X_{n,n}X_{n-1,n}+\theta_n\rho_n(X_{n,n}^2+X_{n-1,n}^2).
\endaligned
\end{equation}
Then, we can write that
\begin{equation}\label{decom-theta-rho}
\left\{\begin{array}{ll}
\frac{1}{b_n}\begin{pmatrix}\sqrt{n\kappa_n^3}(\hat{\theta}_n-\theta^*_n)\\\sqrt{n\kappa_n}(\hat{\rho}_n-\rho^*_n)\end{pmatrix}
=\frac{1}{b_n\sqrt{n\kappa_n}}Y_nZ_n+\frac{1}{b_n}R_{n}(\theta,\rho),& \textrm{Case~I}\\
\frac{1}{\lambda_n}\begin{pmatrix}\sqrt{n/\kappa_n}(\hat{\theta}_n-\theta^*_n)\\\sqrt{n/\kappa_n}(\hat{\rho}_n-\rho^*_n)\end{pmatrix}
=\frac{1}{\lambda_n\sqrt{n\kappa_n}}\tilde{Y}_nM_n+\frac{1}{\lambda_n}\tilde{R}_{n}(\theta,\rho),& \textrm{Case~II}\\
\end{array}\right.,
\end{equation}
where, for all $1\leq l\leq n$,
\begin{align} Z_{l,n}=\begin{pmatrix}\frac{M_{l,n}}{\kappa_n}\\\\U_{l,n}\end{pmatrix},
\qquad Z_{n,n}:=Z_n,
\end{align}
\begin{equation}\label{def-Y-Z}
Y_n=\begin{pmatrix}\frac{n\kappa_n^3}{(1+\theta_n\rho_n)S_{n-1,n}}&0\\\\
\frac{n\kappa_n^2}{J_{n-1,n}}\Big(\frac{1+\theta_n^*\rho_n^*}{1+\theta_n\rho_n}
-\frac{\theta_n^*}{\theta_n}\Big)&
\frac{n\kappa_n\theta^*_n}{\theta_nJ_{n-1,n}}\end{pmatrix},
\end{equation}
\begin{equation}\label{def-tildeY-tildeZ}
\tilde{Y}_n=\begin{pmatrix}\frac{n}{(1+\theta_n\rho_n)S_{n-1,n}}\\\\
\frac{n}{J_{n-1,n}}\Big(\frac{1+\theta^*_n\rho^*_n}{1+\theta_n\rho_n}-\frac{\theta_n^*}{\theta_n}-\frac{H_n}{S_{n-1,n}(1+\theta_n\rho_n)}\Big)\end{pmatrix},
\end{equation}
\begin{equation}\label{def-theta-rho-remainder}
R_{n}(\theta,\rho)=\begin{pmatrix}\frac{\sqrt{n\kappa_n^3}\theta_n\rho_nX_{n,n}X_{n-1,n}}{(1+\theta_n\rho_n)S_{n-1,n}}\\\\
\frac{\sqrt{n\kappa_n}}{J_{n-1,n}}\left(-(\hat\theta_n-\theta^*_n)H_n+\xi_n^I-\rho^*_n\xi_n^J\right)\end{pmatrix},
\end{equation}
and
\begin{equation}\label{def-tildetheta-tilderho-remainder}
\tilde{R}_{n}(\theta,\rho)=\begin{pmatrix}\frac{\sqrt{n/\kappa_n}\theta_n\rho_nX_{n,n}X_{n-1,n}}{(1+\theta_n\rho_n)S_{n-1,n}}\\\\
\frac{\sqrt{n/\kappa_n}}{J_{n-1,n}}\left(\frac{\theta_n^*}{\theta_n}U_n-\frac{\theta_n\rho_n}{1+\theta_n\rho_n}\cdot\frac{X_{n,n}X_{n-1,n}}{S_{n-1,n}}H_n
+\xi_n^I-\rho^*_n\xi_n^J\right)\end{pmatrix}.
\end{equation}


\subsection{Martingales and  predictable quadratic variations}

For $1\leq l\leq n$, denote $\mathcal{F}_{l}=\sigma(V_1,...,V_l)$, then we know that, $(M_{l,n})_{1\leq l\leq n}$, $(N_{l,n})_{1\leq l\leq n}$, and $(U_{l,n})_{1\leq l\leq n}$ are all martingales with respect to the filtration $(\mathcal{F}_{l})_{1\leq l\leq n}$, with predictable quadratic variations
\begin{equation}\label{variation}
 \langle M_{\bullet,n}\rangle_l=\sigma^2S_{l-1,n}, \quad \langle N_{\bullet,n}\rangle_l=\sigma^2S_{l-2,n}, \quad \langle U_{\bullet,n}\rangle_l=\sigma^2T_{l-1,n},
\end{equation}
and predictable quadratic covariations
\begin{equation}\label{covariation}
\langle M_{\bullet,n},N_{\bullet,n}\rangle_l=\sigma^2P_{l-1,n},\qquad \langle M_{\bullet,n},U_{\bullet,n}\rangle_l=\sigma^2Q_{l-1,n}.
\end{equation}

In the following two lemmas, we can show the exponential equivalence of $\frac{M_n}{n}$, $\frac{N_n}{n}$, $\frac{U_n}{n}$
and their predictable quadratic variations in the sense of moderate deviations.
The explicit proofs are postponed to Section 4.

\begin{lem}\label{lem-martingales-ldp}
Let $a_n=b_n$ in (Case~I) and $a_n=\lambda_n$ in (Case~II), we have
\vskip3pt
\noindent(a) for any $\delta>0$,
$$
\lim_{n\to\infty}\frac{1}{a_n^2}\log P\left(\frac{|M_n|}{n}>\delta\right)=-\infty,\quad
\lim_{n\to\infty}\frac{1}{a_n^2}\log P\left(\frac{|N_n|}{n}>\delta\right)=-\infty;
$$
\noindent(b) for any $\delta>0$,
$$
\lim_{n\to\infty}\frac{1}{a_n^2}\log P\left(\frac{|U_n|}{n}>\delta\right)=-\infty.
$$
\end{lem}

\begin{lem}\label{lem-quar-coquar-ldp}
We have
\vskip3pt
\noindent(a) for any $\delta>0$,
$$
\left\{\begin{array}{ll}
\lim_{n\to\infty}\frac{1}{b_n^2}\log P\left(\frac{\max_{1\leq k\leq n}X_{k,n}^2}{b_n\sqrt{n\kappa_n^3}}>\delta\right)=-\infty,& \rm{Case~I}\\
\lim_{n\to\infty}\frac{1}{\lambda_n^2}\log P\left(\frac{\max_{1\leq k\leq n}X_{k,n}^2}{\lambda_n\sqrt{n\kappa_n}}>\delta\right)=-\infty,& \rm{Case~II}\\
\end{array}\right.;
$$
\vskip5pt

\noindent(b) for any $\delta>0$,
$$
\left\{\begin{array}{ll}
\lim_{n\to\infty}\frac{1}{b_n^2}\log P\left(\left|\frac{S_n}{n\kappa_n^3}+\frac{\sigma^2}
{2\gamma_1\gamma_2(\gamma_1+\gamma_2)}\right|>\delta\right)=-\infty,& \rm{Case~I}\\
\lim_{n\to\infty}\frac{1}{\lambda_n^2}\log P\left(\left|\frac{S_n}{n\kappa_n}+\frac{(\gamma_1+\gamma_2)\sigma^2}{8\gamma_1\gamma_2}\right|
>\delta\right)=-\infty,& \rm{Case~II}\\
\end{array}\right.,
$$
and
$$
\left\{\begin{array}{ll}
\lim_{n\to\infty}\frac{1}{b_n^2}\log P\left(\left|\frac{P_n}{n\kappa_n^3}
+\frac{\sigma^2}{2\gamma_1\gamma_2(\gamma_1+\gamma_2)}\right|
>\delta\right)=-\infty,& \rm{Case~I}\\
\lim_{n\to\infty}\frac{1}{\lambda_n^2}\log P\left(\left|\frac{P_n}{n\kappa_n}
-\frac{(\gamma_1-\gamma_2)\sigma^2}{8\gamma_1\gamma_2}\right|
>\delta\right)=-\infty,& \rm{Case~II}\\
\end{array}\right.;
$$
\vskip5pt

\noindent(c) let $a_n=b_n$ in (Case~I) and $a_n=\lambda_n$ in (Case~II), for any $\delta>0$,
$$
\lim_{n\to\infty}\frac{1}{a_n^2}\log P\left(\left|\frac{T_n}{n\kappa_n}
+\frac{\sigma^2}{2\gamma_2}\right|>\delta\right)=-\infty;
$$
\vskip5pt

\noindent(d) for any $\delta>0$,
$$
\left\{\begin{array}{ll}
\lim_{n\to\infty}\frac{1}{b_n^2}\log P\left(\left|\frac{Q_n}{n\kappa_n^2}
-\frac{\sigma^2}{2\gamma_2(\gamma_1+\gamma_2)}\right|>\delta\right)=-\infty,& \rm{Case~I}\\
\lim_{n\to\infty}\frac{1}{\lambda_n^2}\log P\left(\left|\frac{Q_n}{n\kappa_n}
+\frac{\sigma^2}{4\gamma_2}\right|>\delta\right)=-\infty,& \rm{Case~II}\\
\end{array}\right..
$$
\end{lem}


\subsection{Moderate deviations for $(\hat{\theta}_n, \hat{\rho}_n)$}
Similar to the definition of $Z_n$, let
\begin{equation}\label{def-tildeZ}
\tilde{Z}_{l,n}=\begin{pmatrix}M_{l,n}\\\\U_{l,n}\end{pmatrix},
\qquad \tilde{Z}_{n,n}:=\tilde{Z}_n.
\end{equation}

From the decomposition~(\ref{decom-theta-rho}), we first establish the moderate deviations for
$Z_n$ and $\tilde{Z}_n$ according to the methods used in Theorem 4.9 of Bitseki Penda {\it et al.} \cite{Penda}. For the readability of the paper, its proof is postponed to Section 4.
 However, it should be noted that our truncation methods are somewhat different from those of \cite{Penda}.
Next, we deal with $Y_n$ and $\tilde{Y}_n$, and show that they are exponentially equivalent to some deterministic matrices, respectively.
Finally, we can complete the proof of Theorem \ref{MDP-theta-rho} and Theorem \ref{MDP-theta-rho-II} by establishing the asymptotic negligibility of $R_n(\theta,\rho)/b_n$
and $\tilde{R}_n(\theta,\rho)/\lambda_n$ in the sense of moderate deviations.
\vskip5pt

We start with the following key lemma whose proof is postponed to Section 4.

\begin{lem}\label{lem-mdp-Z}
\noindent(a)  Under (Case~I),  $\left\{\frac{Z_n}{b_n\sqrt{n\kappa_n}}, n\geq1\right\}$
satisfies the large deviation principle with speed $b_n^2$ and good rate function $I_Z(x)=\frac{x^{\tau}\Theta^{-1}x}{2}$, where
$$\Theta=\begin{pmatrix}
-\frac{\sigma^4}{2\gamma_1\gamma_2(\gamma_1+\gamma_2)}
& \frac{\sigma^4}{2\gamma_2(\gamma_1+\gamma_2)}
\\
\frac{\sigma^4}{2\gamma_2(\gamma_1+\gamma_2)}&-\frac{\sigma^4}{2\gamma_2}
\end{pmatrix}.
$$

\noindent(b)  Under (Case~II),  $\left\{\frac{\tilde{Z}_n}{\lambda_n\sqrt{n\kappa_n}}, n\geq1\right\}$
satisfies the large deviation principle with speed $\lambda_n^2$ and good rate function
 $J_{\tilde Z}(x)=\frac{x^{\tau}\tilde{\Theta}^{-1}x}{2}$, where
$$
\tilde{\Theta}=\begin{pmatrix}
-\frac{(\gamma_1+\gamma_2)\sigma^4}{8\gamma_1\gamma_2}
& -\frac{\sigma^4}{4\gamma_2}
\\
-\frac{\sigma^4}{4\gamma_2}&-\frac{\sigma^4}{2\gamma_2}
\end{pmatrix}.
$$
\end{lem}

As an application, we can get the
moderate deviations for $\hat{\theta}_n$.

\begin{cor}\label{mdp-theta}
\noindent(a)  Under (Case~I),
$\left\{\frac{\sqrt{n\kappa_n^3}}{b_n}(\hat{\theta}_n-\theta^*_n),\, n\geq1\right\}$
satisfies the large deviation principle with speed $b_n^2$ and good rate function
\begin{align*}
I_{\theta}(x)=-\frac{x^2}{\gamma_1\gamma_2(\gamma_1+\gamma_2)}, \quad x\in\mathbb{R}.
\end{align*}
\vskip5pt
\noindent(b)  Under (Case II),
$\left\{\frac{\sqrt{n/\kappa_n}}{\lambda_n}(\hat{\theta}_n-\theta^*_n),\, n\geq1\right\}$
satisfies the large deviation principle with speed $\lambda_n^2$ and good rate function
\begin{align*}
J_{\theta}(x)=-\frac{(\gamma_1+\gamma_2)^3x^2}{16\gamma_1\gamma_2}, \quad x\in\mathbb{R}.
\end{align*}
\end{cor}

\begin{proof}
\noindent(a). Under (Case~I), from the contraction principle in large deviations (Theorem 4.2.1 in \cite{Dembo}) and Lemma \ref{lem-mdp-Z},
it follows that $\left\{\frac{M_n}{b_n\sqrt{n\kappa_n^3}}, n\geq1\right\}$ satisfies the large deviations with speed $b_n^2$ and the good
rate function
$$
I_M(x)=-\frac{\gamma_1\gamma_2(\gamma_1+\gamma_2)}{\sigma^4}x^2.
$$
Notice that for any $\delta>0$,
\begin{align*}
&P\left(\left|\frac{S_{n,n-1}}{n\kappa_n^3}+\frac{\sigma^2}
{2\gamma_1\gamma_2(\gamma_1+\gamma_2)}\right|>\delta\right)\\
&\leq P\left(\left|\frac{S_n}{n\kappa_n^3}+\frac{\sigma^2}
{2\gamma_1\gamma_2(\gamma_1+\gamma_2)}\right|>\delta/2\right)
+P\left(\left|\frac{X_{n,n}^2}{n\kappa_n^3}\right|>\delta/2\right).
\end{align*}
Together with Lemma \ref{lem-quar-coquar-ldp}~(a),(b) and the fact that
$$
\lim_{n\to\infty}\frac{\sqrt{nk_n^3}}{b_n}\to\infty,
$$
we have
$$
\lim_{n\to\infty}\frac{1}{b_n^2}\log P\left(\left|\frac{S_{n,n-1}}{n\kappa_n^3}+\frac{\sigma^2}
{2\gamma_1\gamma_2(\gamma_1+\gamma_2)}\right|>\delta\right)=-\infty.
$$
Consequently, ~$\left\{\frac{\sqrt{n\kappa_n^3}}{b_n}\frac{1}{1+\theta_n\rho_n}\cdot\frac{M_n}{S_{n-1,n}}, n\geq1\right\}$ satisfies the
large deviations with speed $b_n^2$ and good rate function
$$
I_{\theta}(x)=-\frac{x^2}{\gamma_1\gamma_2(\gamma_1+\gamma_2)},\quad x\in\mathbb{R}.
$$
Therefore, to prove the result, by (\ref{theta-decom}), it is sufficient to show that, for all $\delta>0$,
\begin{equation}\label{eq-2}
\lim_{n\to\infty}\frac{1}{b_n^2}\log P\left(\left|\frac{\theta_n\rho_n\sqrt{n\kappa_n^3}}{(1+\theta_n\rho_n)b_n}\frac{X_{n,n}X_{n-1,n}}{S_{n-1,n}}\right|>\delta\right)=-\infty.
\end{equation}
In fact, we have
\begin{align*}
&P\left(\left|\frac{\theta_n\rho_n\sqrt{n\kappa_n^3}}{(1+\theta_n\rho_n)b_n}\frac{X_{n,n}X_{n-1,n}}{S_{n-1,n}}\right|>\delta\right)\\
&\leq P\left(\left|\frac{S_n}{n\kappa_n^3}
+\frac{\sigma^2}{2\gamma_1\gamma_2(\gamma_1+\gamma_2)}\right|>-\frac{\sigma^2}{4\gamma_1\gamma_2(\gamma_1+\gamma_2)}\right)\\
&\quad\quad\;+P\left(\left|\frac{X_{n,n}X_{n-1,n}}{b_n\sqrt{n\kappa_n^3}}\right|
>-\frac{(1+\theta_n\rho_n)\sigma^2\delta}{4\theta_n\rho_n\gamma_1\gamma_2(\gamma_1+\gamma_2)}\right),
\end{align*}
which achieves the proof of (\ref{eq-2}) by Lemma \ref{lem-quar-coquar-ldp}~(a)-(b).
\vskip5pt

\noindent(b).  Under (Case~II), according to (\ref{theta-decom}), Lemma \ref{lem-quar-coquar-ldp} and Lemma~\ref{lem-mdp-Z},
we can get part (b) of this corollary by the same methods as in the proof of part~(a).
\end{proof}

Now, we turn to analyze the matrix $Y_n$ defined as in the decomposition (\ref{def-Y-Z}). From the equation (B.12) in Bercu \& Pro\"{i}a \cite{Bercu-2013}, it follows that
\begin{equation}\label{eq-J}
J_{n-1,n}=(1+\theta^*_n)(1-\theta^*_n)S_n-\frac{2\theta_n^*}{1+\theta_n\rho_n}M_n-(\hat{\theta}_n-\theta_n^*)G_n+\xi_n^J.
\end{equation}
Moreover,
\begin{equation}\label{eq-18-1}
\begin{aligned}
G_n&=2P_n-2\theta_n^*S_n-(\hat{\theta}_{n}-\theta^*_n)S_n\\
&=2\left(\frac{1}{1+\theta_n\rho_n}M_n+\frac{\theta_n\rho_n}{1+\theta_n\rho_n}X_{n,n}X_{n-1,n}-\theta_n^*X_{n,n}^2\right)
-(\hat{\theta}_{n}-\theta^*_n)S_n.
\end{aligned}
\end{equation}

\begin{lem}\label{equiv-J}
Letting $a_n=b_n$ in (Case~I) and $a_n=\lambda_n$ in (Case~II), then
\vskip3pt
\noindent(a) for any $\delta>0$, $\lim_{n\to\infty}\frac{1}{a_n^2}\log P\left(\frac{|\xi_{n}^{J}|}{{n\kappa_n}}>\delta\right)=-\infty$;
\vskip5pt

\noindent(b) for any $\delta>0$,
$\lim_{n\to\infty}\frac{1}{a_n^2}\log P\left(\left|\frac{J_{n-1,n}}{n\kappa_n}
+\frac{\sigma^2}{2(\gamma_1+\gamma_2)}\right|>\delta\right)=-\infty$.
\end{lem}

\begin{proof}
Since
$$
\left\{\begin{array}{ll}
\lim_{n\to\infty}\frac{n\kappa_n}{b_n\sqrt{n\kappa_n^3}}=\infty,& \textrm{Case~I}\\
\lim_{n\to\infty}\frac{n\kappa_n}{\lambda_n\sqrt{n\kappa_n}}=\infty,& \textrm{Case~II}\\
\end{array}\right.,
$$
it follows from Lemma \ref{lem-quar-coquar-ldp}~(a) that
\begin{equation}\label{eq-2-1}
\lim_{n\to\infty}\frac{1}{a_n^2}\log P\left(\frac{|X_{n,n}X_{n-1,n}|\vee X_{n,n}^2}{{n\kappa_n}}>\delta\right)=-\infty,
\end{equation}
which implies that
$$
\lim_{n\to\infty}\frac{1}{a_n^2}\log P\left(\frac{|\xi_{n}^{P}|}{{n\kappa_n}}>\delta\right)=-\infty.
$$
Therefore, by the definition of $\xi_n^J$ (refer to (\ref{xiIJPQ})), to prove part (a), it is enough to show that
\begin{equation}\label{eq-3}
\lim_{n\to\infty}\frac{1}{a_n^2}\log P\left(\frac{|(\hat{\theta}_n-\theta_n^*)X_{n,n}X_{n-1,n}|}{{n\kappa_n}}>\delta\right)=-\infty
\end{equation}
and
\begin{equation}\label{eq-4}
\lim_{n\to\infty}\frac{1}{a_n^2}\log P\left(\frac{|(\hat{\theta}_n^2-(\theta^*_n)^2)(X_{n,n}^2+X_{n-1,n}^2)|}{{n\kappa_n}}>\delta\right)=-\infty.
\end{equation}

We only prove (\ref{eq-3}) and (\ref{eq-4}) under (Case~I), and the proof under (Case~II) is similar and omitted here.
In fact, for any $L>0$, under (Case~I), one can see that
\begin{align*}
&P\left(\frac{|(\hat{\theta}_n-\theta_n^*)X_{n,n}X_{n-1,n}|}{{n\kappa_n}}>\delta\right)\\
&\leq P\left(\frac{\sqrt{n\kappa_n^3}|\hat{\theta}_n-\theta_n^*|}{b_n}>L\right)+
P\left(\frac{b_n|X_{n,n}X_{n-1,n}|}{n\kappa_n\sqrt{n\kappa_n^3}}>\frac{\delta}{L}\right)
\end{align*}
According to (\ref{eq-2-1}), we can write that
$$
\lim_{n\to\infty}\frac{1}{b_n^2}\log P\left(\frac{b_n|X_{n,n}X_{n-1,n}|}{n\kappa_n\sqrt{n\kappa_n^3}}>\frac{\delta}{L}\right)=-\infty
$$
and
$$
\lim_{n\to\infty}\frac{1}{b_n^2}\log P\left(\frac{\sqrt{n\kappa_n^3}|\hat{\theta}_n-\theta_n^*|}{b_n}>L\right)
=\frac{L^2}{\gamma_1\gamma_2(\gamma_1+\gamma_2)}.
$$
Then we complete the proof of (\ref{eq-3}) by taking $L\to+\infty$.
Moreover, by the Delta method in moderate deviations \cite{Gao-Zhao}
and Corollary \ref{mdp-theta}, $\left\{\frac{\sqrt{n\kappa_n^3}}{b_n}(\hat{\theta}_n^2-(\theta^*_n)^2), n\geq1\right\}$ satisfies the large deviation principle with speed $b_n^2$ and the good rate function
$$
\tilde{I}_{\theta}(x)=-\frac{x^2}{4\gamma_1\gamma_2(\gamma_1+\gamma_2)},\quad
x\in\mathbb{R}.
$$
Hence, similar to the proof of (\ref{eq-3}), we can get (\ref{eq-4}) immediately.
\vskip5pt

As for part (b), we can deduce that
$$
(1+\theta^*_n)(1-\theta^*_n)
=\left\{\begin{array}{ll}
\frac{\gamma_1\gamma_2}{\kappa_n^2}+o(\kappa_n^{-2}),& \textrm{Case~I}\\
\frac{4\gamma_1\gamma_2}{(\gamma_1+\gamma_2)^2}+o(\kappa_n^{-1}),& \textrm{Case~II}\\
\end{array}\right..
$$
Therefore, using Lemma~\ref{lem-quar-coquar-ldp}~(b),
\begin{equation}\label{eq-56}
\lim_{n\to\infty}\frac{1}{a_n^2}\log P\left(\left|\frac{(1+\theta^*_n)(1-\theta^*_n)}{n\kappa_n}S_n
+\frac{\sigma^2}{2(\gamma_1+\gamma_2)}\right|>\delta\right)=-\infty.
\end{equation}
By Lemma \ref{lem-martingales-ldp}, (\ref{eq-J}), (\ref{eq-56}) and part (a) of this lemma, it is sufficient to show that, for all $\delta>0$,
\begin{equation}\label{eq-57}
\lim_{n\to\infty}\frac{1}{a_n^2}\log P\left(\frac{|(\hat{\theta}_n-\theta_n^*)G_n|}{{n\kappa_n}}>\delta\right)=-\infty.
\end{equation}
Indeed, for any $L>0$
\begin{align*}
&P\left(\frac{(\hat{\theta}_{n}-\theta^*_n)^2S_n}{{n\kappa_n}}>\delta\right)\\
&\leq\left\{\begin{array}{ll}
 P\left(\frac{n\kappa_n^3(\hat{\theta}_{n}-\theta^*_n)^2}{b_n^2}>L\right)
+P\left(\frac{b_n^2S_n}{n^2\kappa_n^4}>\frac{\delta}{L}\right),& \textrm{Case~I}\\
 P\left(\frac{n\kappa_n(\hat{\theta}_{n}-\theta^*_n)^2}{\lambda_n^2}>L\right)
+P\left(\frac{\lambda_n^2S_n}{n^2\kappa_n^2}>\frac{\delta}{L}\right),& \textrm{Case~II}\\
\end{array}\right..
\end{align*}
So, letting $n\to+\infty$ and then $L\to+\infty$, by
Lemma \ref{lem-quar-coquar-ldp} and Corollary \ref{mdp-theta}, we can obtain that
\begin{equation}\label{eq-5}
\lim_{n\to\infty}\frac{1}{a_n^2}\log P\left(\frac{(\hat{\theta}_{n}-\theta^*_n)^2S_n}{{n\kappa_n}}>\delta\right)
=-\infty.
\end{equation}
Similarly, One can easily see that
\begin{equation}\label{eq-6}
\begin{aligned}
&\lim_{n\to\infty}\frac{1}{a_n^2}\log P\left(\frac{\left|
\left(\frac{1}{1+\theta_n\rho_n}M_n+\frac{\theta_n\rho_n}{1+\theta_n\rho_n}X_{n,n}X_{n-1,n}-\theta_n^*X_{n,n}^2\right)\right|}{{n\kappa_n}}
|\hat{\theta}_n-\theta_n^*|>\delta\right)\\
&=-\infty.
\end{aligned}
\end{equation}
Thus, (\ref{eq-57}) can be achieved by (\ref{eq-5})-(\ref{eq-6}).
\end{proof}

\begin{lem}\label{quar-H}
For $H_n$ defined by (\ref{xiIJPQ}), we have the following results:
$$
\left\{\begin{array}{ll}
\lim_{n\to\infty}\frac{1}{b_n^2}\log P\left(\left|\frac{H_n}{n}-\frac{\sigma^2}{2}\right|>\delta\right)=-\infty,& \textrm{Case~I}\\
\lim_{n\to\infty}\frac{1}{\lambda_n^2}\log P\left(\left|\frac{H_n}{n\kappa_n}+\frac{\sigma^2}{\gamma_1+\gamma_2}\right|
>\delta\right)=-\infty,& \textrm{Case~II}\\
\end{array}\right..
$$
\end{lem}

\begin{proof}
According to (\ref{decom-S-P}), we can write that
\begin{align*}
W_n&:=\sum_{k=2}^n{X_{k,n}X_{k-2,n}}\\
&=(\theta_n+\rho_n)P_n-\theta_n\rho_nS_n+N_n-(\theta_n+\rho_n)X_{n,n}X_{n-1,n}+\theta_n\rho_n(X_{n,n}+X_{n-1,n})\\
&=\left((\theta_n+\rho_n)\theta^*_n-\theta_n\rho_n\right)S_n+R_{n2},
\end{align*}
where $R_{n2}$ is the remainder term.
Consequently, (\ref{eq-18-1}) implies that
\begin{equation}\label{eq-H}
\begin{aligned}
H_n&=F_n-\rho^*_nG_n\\
&=\left(S_n+Q_n-(\hat{\theta}_n+\theta^*_n)P_n\right)-\rho^*_n\left(2P_n-(\hat{\theta}_n+\theta^*_n)S_n\right)\\
&=\frac{(1-\theta_n\rho_n)(1-\theta_n)(1-\rho_n)(1+\theta_n^*)}{1+\theta_n\rho_n}S_n
+(\hat{\theta}_n-\theta_n^*)(\rho_n^*S_n-P_n)+R_{n3},
\end{aligned}
\end{equation}
where $R_{n3}$ is the corresponding remainder.
Applying Lemmas \ref{lem-martingales-ldp}, \ref{lem-quar-coquar-ldp}
and Corollary \ref{mdp-theta}, we have, for any $\delta>0$,
\begin{equation}\label{eq-H-1}
\left\{\begin{array}{ll}
\lim_{n\to\infty}\frac{1}{b_n^2}\log P\left(\frac{1}{n}\left|(\hat{\theta}_n-\theta_n^*)(\rho_n^*S_n-P_n)+R_{n3}\right|\geq\delta\right)=-\infty,& \textrm{Case~I}\\
\lim_{n\to\infty}\frac{1}{\lambda_n^2}\log P\left(\frac{1}{n\kappa_n}\left|(\hat{\theta}_n-\theta_n^*)(\rho_n^*S_n-P_n)+R_{n3}\right|\geq\delta\right)=-\infty,& \textrm{Case~II}\\
\end{array}\right..
\end{equation}
Furthermore, by some simple calculations,
$$
\frac{(1-\theta_n\rho_n)(1-\theta_n)(1-\rho_n)(1+\theta_n^*)}{1+\theta_n\rho_n}
=\left\{\begin{array}{ll}
-\frac{\gamma_1\gamma_2(\gamma_1+\gamma_2)}{\kappa_n^3}+O(\kappa_n^{-4}),& \textrm{Case~I}\\
\frac{8\gamma_1\gamma_2}{(\gamma_1+\gamma_2)^2}+O(\kappa_n^{-1}),& \textrm{Case~II}\\
\end{array}\right..
$$
Combined with Lemma~\ref{lem-quar-coquar-ldp}, (\ref{eq-H}) and (\ref{eq-H-1}),
we complete the proof of the lemma.
\end{proof}

\begin{lem}\label{equiv-Y}
Letting $Y_n$ and $\tilde{Y}_n$ defined by (\ref{def-Y-Z}) and (\ref{def-tildeY-tildeZ}) respectively,  we have for any $\delta>0$,
$$
\left\{\begin{array}{ll}
\lim_{n\to\infty}\frac{1}{b_n^2}\log P\left(\|Y_n-\Upsilon\|>\delta\right)=-\infty,& \textrm{Case~I}\\
\lim_{n\to\infty}\frac{1}{\lambda_n^2}\log P\left(\|\tilde{Y}_n-\tilde\Upsilon\|>\delta\right)=-\infty,& \textrm{Case~II}\\
\end{array}\right.,
$$
where
$$
\Upsilon:=\begin{pmatrix}
-\frac{\gamma_1\gamma_2(\gamma_1+\gamma_2)}{\sigma^2} &0
\\
-\frac{2\gamma_1(\gamma_1+\gamma_2)}{\sigma^2}&-\frac{2(\gamma_1+\gamma_2)}{\sigma^2}
\end{pmatrix},\qquad
\tilde\Upsilon:=\begin{pmatrix}\frac{8\gamma_1\gamma_2}{(\gamma_1+\gamma_2)^2\sigma^2}\\\\
-\frac{8\gamma_1\gamma_2}{(\gamma_1+\gamma_2)^2\sigma^2}\end{pmatrix}.
$$
\end{lem}

\begin{proof}
Under (Case~II), since
$$
1+\theta_n\rho_n=-\frac{\gamma_1+\gamma_2}{\kappa_n}-\frac{\gamma_1\gamma_2}{\kappa_n^2},
$$
according to Lemma~\ref{quar-H}, we have, for any $\delta>0$
\begin{equation}\label{equiv-Y-eq1}
\lim_{n\to\infty}\frac{1}{\lambda_n^2}\log P\left(\left|\frac{H_n}{S_{n-1,n}(1+\theta_n\rho_n)\kappa_n}
+\frac{8\gamma_1\gamma_2}{(\gamma_1+\gamma_2)^3}\right|
>\delta\right)=-\infty.
\end{equation}
Moreover, note that
$$
\frac{\theta_n^*}{\theta_n}=
\left\{\begin{array}{ll}
1-\frac{\gamma_1}{\kappa_n}+O(\kappa_n^{-2}),& \textrm{Case~I}\\
\frac{\gamma_2-\gamma_1}{\gamma_1+\gamma_2}+O(\kappa_n^{-1}),& \textrm{Case~II}\\
\end{array}\right.
$$
and
$$
\frac{1+\theta_n^*\rho_n^*}{1+\theta_n\rho_n}=
\left\{\begin{array}{ll}
1+O(1/{\kappa_n^2}),& \textrm{Case~I}\\
-\frac{4\gamma_1\gamma_2}{(\gamma_1+\gamma_2)^3}\kappa_n+O(\kappa_n^{-1}),& \textrm{Case~II}\\
\end{array}\right..
$$
Then, together with Lemma \ref{lem-quar-coquar-ldp}, Lemma \ref{equiv-J} and (\ref{equiv-Y-eq1}), we can complete the proof of the lemma.
\end{proof}

We are now in a position to prove the main results, Theorems \ref{MDP-theta-rho} and \ref{MDP-theta-rho-II}.
\vskip5pt

\noindent{\it Proof of Theorems \ref{MDP-theta-rho} and \ref{MDP-theta-rho-II}.}
Under (Case~I), from Lemmas \ref{lem-mdp-Z} and \ref{equiv-Y}, we know that the sequence $\left\{\frac{1}{b_n\sqrt{n\kappa_n}}Y_nZ_n, n\geq1\right\}$ satisfies the
large deviation principle with speed $b_n^2$ and the rate function
$$
I_{Y,Z}(x)=-\frac{x^{\tau}(\Upsilon\Theta\Upsilon^{\tau})^{-1} x}{2},\quad x\in\mathbb{R}^2.
$$
Note that
$$
\Upsilon\Theta\Upsilon^{\tau}=
\begin{pmatrix}
-\frac{\gamma_1\gamma_2(\gamma_1+\gamma_2)}{2} & 0
\\
0&-2(\gamma_1+\gamma_2)
\end{pmatrix}=\Gamma,
$$
hence, we have $I_{Y,Z}(x)=I_{\theta,\rho}(x)$.
\vskip5pt

On the other hand, under (Case~II), using  Lemmas \ref{lem-mdp-Z} and \ref{equiv-Y} again, each component of
$\left\{\frac{1}{\lambda_n\sqrt{n\kappa_n}}\tilde{Y}_nM_n, n\geq1\right\}$ satisfies the
large deviation principle with speed $\lambda_n^2$ and the rate function
$$
J(x)=-\frac{(\gamma_1+\gamma_2)^3}{16\gamma_1\gamma_2}x^2,\quad x\in\mathbb{R}.
$$
\vskip5pt

Recalling the decomposition (\ref{decom-theta-rho}), to obtain our main result, we only need to prove that
\begin{equation}\label{equiv-remainder-0}
\left\{\begin{array}{ll}
\lim_{n\to\infty}\frac{1}{b_n^2}\log P\left(\frac{|R_n(\theta,\rho)|}{b_n}>\delta\right)=-\infty,& \textrm{Case~I}\\
\lim_{n\to\infty}\frac{1}{\lambda_n^2}\log P\left(\frac{|\tilde{R}_n(\theta,\rho)|}{\lambda_n}>\delta\right)=-\infty,& \textrm{Case~II}\\
\end{array}\right..
\end{equation}

In fact, under (Case~I), similar to the proof of (\ref{eq-3}), we can obtain by Lemma~\ref{quar-H} that
\begin{equation}\label{equiv-remainder-1}
\lim_{n\to\infty}\frac{1}{b_n^2}\log P\left(\frac{|(\hat{\theta}_n-\theta^*_n)H_n|}{b_n\sqrt{n\kappa_n}}>\delta\right)=-\infty.
\end{equation}
Moreover, under (Case~II), by Lemma~\ref{lem-mdp-Z} and (\ref{equiv-Y-eq1}),
\begin{equation}\label{equiv-remainder-1-1}
\lim_{n\to\infty}\frac{1}{\lambda_n^2}\log P\left(\frac{\left|\frac{\theta_n^*}{\theta_n}U_n-\frac{\theta_n\rho_n}{1+\theta_n\rho_n}\cdot\frac{X_{n,n}X_{n-1,n}}{S_{n-1,n}}H_n\right|}
{\lambda_n\sqrt{n\kappa_n^3}}>\delta\right)=-\infty.
\end{equation}
Therefore, to prove (\ref{equiv-remainder-0}), by Lemma~\ref{equiv-J}, (\ref{equiv-remainder-1}) and (\ref{equiv-remainder-1-1}),
 we need to show the following result:
\begin{equation}\label{equiv-remainder-2}
\left\{\begin{array}{ll}
\lim_{n\to\infty}\frac{1}{b_n^2}\log P\left(\frac{|\xi_n^I-\rho^*_n\xi_n^J|}{b_n\sqrt{n\kappa_n}}>\delta\right)=-\infty,& \textrm{Case~I}\\
\lim_{n\to\infty}\frac{1}{\lambda_n^2}\log P\left(\frac{|\xi_n^I-\rho^*_n\xi_n^J|}{\lambda_n\sqrt{n\kappa_n^3}}>\delta\right)=-\infty,& \textrm{Case~II}\\
\end{array}\right..
\end{equation}

\begin{proof}[Proof of (\ref{equiv-remainder-2})]
By straightforward calculations, we get that
\begin{align}
\xi_n^I-\rho^*_n\xi_n^J=\Delta_{n1}+\Delta_{n2},
\end{align}
where
\begin{align*}
\Delta_{n1}&:=
\frac{\theta_n\rho_n(1-\theta_n^2)(1-\rho_n^2)}{(1+\theta_n\rho_n)^3}X_{n,n}X_{n-1,n}
+\rho^*_n(\theta^*_n+1)(\theta^*_n-1)X_{n-1,n}^2
\end{align*}
and
\begin{align*}
\Delta_{n2}&:=\left((\theta^*_n)^2-\hat{\theta}_n^2+2\rho^*_n(\theta^*_n-\hat{\theta}_n)\right)X_{n,n}X_{n-1,n}\\
&\;\quad +\left(\rho^*_n(\hat{\theta}_n^2-(\theta^*_n)^2)+(\hat{\theta}_n-\theta^*_n)\right)X_{n,n}^2
+\rho^*_n\left(\hat{\theta}_n^2-(\theta^*_n)^2\right)X_{n-1,n}^2.
\end{align*}
Some simple calculations imply that
$$
\frac{\theta_n\rho_n(1-\theta_n^2)(1-\rho_n^2)}{(1+\theta_n\rho_n)^3}=
\left\{\begin{array}{ll}
\frac{\gamma_1\gamma_2}{2\kappa_n^2}+O\left(\kappa_n^{-3}\right),& \textrm{Case~I}\\
\frac{4\gamma_1\gamma_2}{(\gamma_1+\gamma_2)^3}\kappa_n+O(1),& \textrm{Case~II}\\
\end{array}\right.
$$
and
$$
\rho^*_n(\theta^*_n+1)(\theta^*_n-1)=
\left\{\begin{array}{ll}
-\frac{\gamma_1\gamma_2}{\kappa_n^2}+O\left(\kappa_n^{-3}\right),& \textrm{Case~I}\\
\frac{4\gamma_1\gamma_2(\gamma_2-\gamma_1)}{(\gamma_1+\gamma_2)^3}+O(\kappa_n^{-1}),& \textrm{Case~II}\\
\end{array}\right..
$$
Hence, it follows that, from Lemma \ref{lem-quar-coquar-ldp},
$$
\left\{\begin{array}{ll}
\lim_{n\to\infty}\frac{1}{b_n^2}\log P\left(\frac{|\Delta_{n1}|}{b_n\sqrt{n\kappa_n}}>\delta\right)=-\infty,& \textrm{Case~I}\\
\lim_{n\to\infty}\frac{1}{\lambda_n^2}\log P\left(\frac{|\Delta_{n1}|}{\lambda_n\sqrt{n\kappa_n^3}}>\delta\right)=-\infty,& \textrm{Case~II}\\
\end{array}\right..
$$
Moreover, by Lemma \ref{lem-quar-coquar-ldp}, Corollary \ref{mdp-theta} and the similar methods as in the proof of (\ref{eq-4}), we also get
$$
\left\{\begin{array}{ll}
\lim_{n\to\infty}\frac{1}{b_n^2}\log P\left(\frac{|\Delta_{n2}|}{b_n\sqrt{n\kappa_n}}>\delta\right)=-\infty,& \textrm{Case~I}\\
\lim_{n\to\infty}\frac{1}{\lambda_n^2}\log P\left(\frac{|\Delta_{n2}|}{\lambda_n\sqrt{n\kappa_n^3}}>\delta\right)=-\infty,& \textrm{Case~II}\\
\end{array}\right..
$$
That is to say, we obtain the equation (\ref{equiv-remainder-2}). Hence we complete the proof of Theorems \ref{MDP-theta-rho} and \ref{MDP-theta-rho-II}.
\end{proof}


\subsection{Moderate deviations for $\hat{d}_n$}

For the sake of the readers, let us recall some notations defined in previous sections,
\[
J_n:=J_{n,n}:=\sum_{k=1}^n\hat{\varepsilon}_{k,n}^2,\quad \hat{\varepsilon}_{n}^2:=\hat{\varepsilon}_{n,n}^2,\quad \hat{\varepsilon}_{k,n}^2:=X_{k,n}-\hat{\theta}_nX_{k-1,n}
\]
and the Durbin-Watson statistic,
\[
\hat{d}_n:=\frac{\sum_{k=1}^n(\hat{\varepsilon}_{k,n}-\hat{\varepsilon}_{k-1,n})^2}
{\sum_{k=1}^{n}\hat{\varepsilon}_{k,n}^2},\quad
d^*:=2(1-\rho_n^*)=2(1-\theta_n\rho_n\theta_n^*).
\]
Putting $f_n:=\hat{\varepsilon}_n^2/J_n$, from the equation (C.4) in Bercu \& Pro\"{i}a \cite{Bercu-2013}, we know that
\begin{equation}\label{eq-D}
\hat{d}_n-d^*_n=-2(\hat{\rho}_n-\rho^*_n)+R_{n}(d),
\end{equation}
where, the remainder term $R_{n}(d)=2(\hat{\rho}_n-\rho^*_n)f_n+(2\rho^*_n-1)f_n$.
\vskip5pt

\begin{proof}[Proof of Corollary \ref{mdp-DW}] By (\ref{eq-D}), Theorems \ref{MDP-theta-rho} and \ref{MDP-theta-rho-II},
it is sufficient to show the asymptotic negligibility of $R_n(d)$, i.e. for any $\delta>0$,
\begin{equation}\label{R-d}
\left\{\begin{array}{ll}
\lim_{n\to\infty}\frac{1}{b_n^2}\log P\left(\frac{\sqrt{n\kappa_n}|R_n(d)|}{b_n}>\delta\right)=-\infty,& \textrm{Case~I}\\
\lim_{n\to\infty}\frac{1}{\lambda_n^2}\log P\left(\frac{\sqrt{n/\kappa_n}|R_n(d)|}{\lambda_n}>\delta\right)=-\infty,& \textrm{Case~II}\\
\end{array}\right..
\end{equation}

We only prove (\ref{R-d}) under (Case~I), while the proof under (Case~II) is similar and omitted here. In fact,
\begin{align*}
f_n
&=\frac{\left((\theta_n^*-\hat{\theta}_n)X_{k-1,n}
+(\theta_n-\theta_n^*)X_{n-1,n}+\varepsilon_{n,n}\right)^2}{J_n}\\
&\leq\frac{4}{J_n}\left((\theta_n^*-\hat{\theta}_n)^2X_{n-1,n}^2
+(\theta_n-\theta_n^*)^2X^2_{n-1,n}+\varepsilon^2_{n,n}\right),
\end{align*}
therefore, for any $L>0$, we have
\begin{align*}
&P\left(\frac{\sqrt{n\kappa_n}|f_n|}{b_n}>\delta\right)\\
&\leq P\left(\left|\frac{J_{n}}{n\kappa_n}
+\frac{\sigma^2}{2(\gamma_1+\gamma_2)}\right|>-\frac{\sigma^2}{4(\gamma_1+\gamma_2)}\right)\\
&\quad+P\left(\frac{(\theta_n^*-\hat{\theta}_n)^2X_{n-1,n}^2+(\theta_n-\theta_n^*)^2X^2_{n-1,n}+\varepsilon^2_{n,n}}{b_n\sqrt{n\kappa_n}}
>-\frac{\sigma^2}{16(\gamma_1+\gamma_2)}\right)\\
&\leq P\left(\left|\frac{J_{n}}{n\kappa_n}
+\frac{\sigma^2}{2(\gamma_1+\gamma_2)}\right|>-\frac{\sigma^2}{4(\gamma_1+\gamma_2)}\right)
+P\left(\frac{\varepsilon^2_{n,n}}{b_n\sqrt{n\kappa_n}}>-\frac{\sigma^2}{48(\gamma_1+\gamma_2)}\right)\\
&\quad+P\left(\frac{X^2_{n-1,n}}{b_n\sqrt{n\kappa_n^3}}>-\frac{\sigma^2}{48(\gamma_1+\gamma_2)\kappa_n(\theta_n-\theta_n^*)^2}\right)
+P\left(\frac{n\kappa_n^3}{b_n^2}(\theta_n^*-\hat{\theta}_n)^2>L\right)\\
&\quad+P\left(\frac{X^2_{n-1,n}}{b_n\sqrt{n\kappa_n^3}}>-\frac{n\kappa_n^2\sigma^2}{48(\gamma_1+\gamma_2)Lb_n^2}\right).
\end{align*}
Letting $n\to+\infty$ and then $L\to+\infty$, we can obtain that, by Lemma \ref{lem-quar-coquar-ldp}, the equation (\ref{deviation-ineq-max-varepsilon}), Corollary \ref{mdp-theta} and Lemma \ref{equiv-J},
$$
\lim_{n\to\infty}\frac{1}{b_n^2}\log P\left(\frac{\sqrt{n\kappa_n}|f_n|}{b_n}>\delta\right)=-\infty.
$$
We are done.
\end{proof}



\section{Technical appendix}\label{sec5}
In this section, we will give explicit proofs to Lemma \ref{lem-martingales-ldp}, Lemma \ref{lem-quar-coquar-ldp} and Lemma \ref{lem-mdp-Z}.

\subsection{Proofs of Lemma \ref{lem-martingales-ldp} and Lemma \ref{lem-quar-coquar-ldp}}

We first recall a result of large deviations for i.i.d. random variables.

\begin{lem}[Eichelsbacher \& L\"{o}we \cite{Eichelsbacher}]\label{mdp-V}
Let
\begin{align}
L_n=\sum_{k=1}^nV_k^2\quad {\rm and}\quad \Lambda_n=\sum_{k=1}^nV_k^4.
\end{align}
Under (Case~I) and (Case~II),
 $\left\{\frac{L_n-n\sigma^2}{\sqrt{n}a_n}, n\geq1\right\}$
and $\left\{\frac{\Lambda_n-nEV_1^4}{\sqrt{n}a_n}, n\geq1\right\}$ satisfy the large deviation principle with speed $a_n^2$ and rate functions
\begin{align}
I_L(x)=\frac{x^2}{2E(V_1^2-\sigma^2)^2}\quad{\rm and}\quad I_{\Lambda}(x)=\frac{x^2}{2E(V_1^4-EV_1^4)^2},
\end{align}
respectively, where $a_n=b_n$ or $\lambda_n$.
\end{lem}

\noindent{\it\textbf{Proof of Lemma \ref{lem-martingales-ldp}}.}
For part (a), we only need to prove the result for $M_n$ because the asymptotic negligibility of $\frac{N_n}{n}$ can be obtained similarly.
Since $(M_{l,n})_{1\leq l\leq n}$ is a locally square integrable martingale, we infer that, from Theorem 2.1 of Bercu \& Touati \cite{Bercu-2008}, for all $x,\,y>0$,
\begin{equation}\label{deviation-ineq-martingale-quar-M}
P\left(|M_n|>x, \langle M_{\bullet,n}\rangle_n+[M_{\bullet,n}]_n\leq y\right)\leq2\exp\left\{-\frac{x^2}{2y}\right\},
\end{equation}
where the total quadratic variation $[M_{\bullet,n}]_n=\sum_{k=1}^nX_{k-1,n}^2V_k^2$.
By (\ref{deviation-ineq-martingale-quar-M}), we have, for all $\delta>0$,
\begin{align*}
P\left(\frac{|M_n|}{n}>\delta\right)&\leq P\left(|M_n|>\delta n, \langle M_{\bullet,n}\rangle_n+[M_{\bullet,n}]_n\leq n\kappa_n^5\right)\\
&\quad\quad+P\left(\langle M_{\bullet,n}\rangle_n+[M_{\bullet,n}]_n> n\kappa_n^5\right)\\
&\leq2\exp\left\{-\frac{\delta^2n}{2 \kappa_n^5}\right\}+P\left(\langle M_{\bullet,n}\rangle_n>\frac{n\kappa_n^5}{2}\right)
+P\left([M_{\bullet,n}]_n>\frac{n\kappa_n^5}{2}\right).\\
\end{align*}
Obviously, by the assumptions,
\[
\lim_{n\to\infty}\frac{1}{a_n^2}\log\left(2\exp\left\{-\frac{\delta^2n}{2 \kappa_n^5}\right\}\right)=-\infty.
\]
We only need to deal with the last two terms. According to (A.8) and (A.9) in Bercu \& Pro\"{i}a~\cite{Bercu-2013}, we have, for any $a>0$,
\begin{equation}\label{Sn-Vn}
\sum_{k=1}^n|X_{k,n}|^a\leq\left(1-|\theta_n|\right)^{-a}\left(1-|\rho_n|\right)^{-a}\sum_{k=1}^n|V_{k}|^a.
\end{equation}
Put $\Xi_{l,n}:=\sum_{k=1}^{l}X_{k,n}^4$ and $\Xi_n:=\Xi_{n,n}$, then by (\ref{Sn-Vn}),
\begin{equation}\label{quar-coquar-ineq-M}
\langle M_{\bullet,n}\rangle_n=\sigma^2 S_{n-1,n}\leq\frac{\sigma^2\kappa_n^4}{\gamma_1^2\gamma_2^2}L_n\quad{\rm and}\quad
[M_{\bullet,n}]_n\leq \Xi_n^{\frac{1}{2}}\Lambda_n^{\frac{1}{2}}\leq\frac{\kappa_n^4}{\gamma_1^2\gamma_2^2}\Lambda_n.
\end{equation}
Therefore,
{\small\begin{align*}
&P\left(\langle M_{\bullet,n}\rangle_n>\frac{n\kappa_n^5}{2}\right)
+P\left([M_{\bullet,n}]_n>\frac{n\kappa_n^5}{2}\right)\\
&\leq P\left(\frac{L_n-n\sigma^2}{\sqrt{n}a_n}>\frac{\gamma_1^2\gamma_2^2\sqrt{n}\kappa_n}{2\sigma^2 a_n}-\frac{\sqrt{n}}{a_n}\sigma^2\right)
+P\left(\frac{\Lambda_n-nEV_1^4}{\sqrt{n}a_n}>\frac{\gamma_1^2\gamma_2^2\sqrt{n}\kappa_n}{a_n}-\frac{\sqrt{n}}{a_n}EV_1^4\right).
\end{align*}}
From Lemma \ref{mdp-V} and the fact that, as $n\to\infty$,
$$
\frac{\gamma_1^2\gamma_2^2\sqrt{n}\kappa_n}{2\sigma^2 a_n}-\frac{\sqrt{n}}{a_n}\sigma^2\to\infty,\quad
\frac{\gamma_1^2\gamma_2^2\sqrt{n}\kappa_n}{a_n}-\frac{\sqrt{n}}{a_n}EV_1^4\to\infty,
$$
we get that
$$
\lim_{n\to\infty}\frac{1}{a_n^2}\left(\log P\left(\langle M_{\bullet,n}\rangle_n>\frac{n\kappa_n^5}{2}\right)
\vee\log P\left([M_{\bullet,n}]_n>\frac{n\kappa_n^5}{2}\right)\right)=-\infty,
$$
which achieves the proof of part (a).
\vskip5pt

For part (b), applying (\ref{Sn-Vn}), we can write that
\begin{equation}\label{quar-coquar-ineq-U}
\langle U_{\bullet,n}\rangle_n=\sigma^2\sum_{k=1}^{n-1}\varepsilon_{k,n}^2
\leq\frac{\sigma^2\kappa_n^2}{\gamma_2^2}L_n \quad{\rm and}\quad
[U_{\bullet,n}]_n=\sum_{k=1}^{n}\varepsilon_{k-1,n}^2V_k^2
\leq\frac{\kappa_n^2}{\gamma_2^2}\Lambda_n.
\end{equation}
Hence, similar to the proof of part (a) in Lemma \ref{lem-martingales-ldp}, we have
{\small\begin{align*}
P\left(\frac{|U_n|}{n}>\delta\right)&\leq
2\exp\left\{-\frac{\delta^2n}{2 \kappa_n^3}\right\}+P\left(\langle U_{\bullet,n}\rangle_n>\frac{n\kappa_n^3}{2}\right)
+P\left([U_{\bullet,n}]_n>\frac{n\kappa_n^3}{2}\right)
\end{align*}}
and
\begin{align*}
&P\left(\langle U_{\bullet,n}\rangle_n>\frac{n\kappa_n^3}{2}\right)
+P\left([U_{\bullet,n}]_n>\frac{n\kappa_n^3}{2}\right)\\
&\leq P\left(\frac{L_n-n\sigma^2}{\sqrt{n}a_n}>\frac{\sqrt{n}}{a_n}\left(\frac{\gamma_2^2\kappa_n}{2\sigma^2}-\sigma^2\right)\right)
+P\left(\frac{\Lambda_n-nEV_1^4}{\sqrt{n}a_n}>\frac{\sqrt{n}}{a_n}\left(\gamma_2^2\kappa_n-EV_1^4\right)\right),
\end{align*}
which, together with Lemma \ref{mdp-V}, yield that
$$
\lim_{n\to\infty}\frac{1}{a_n^2}\log P\left(\frac{|U_n|}{n}>\delta\right)=-\infty.
$$

\noindent{\it\textbf{Proof of Lemma \ref{lem-quar-coquar-ldp}}.}
For part (a), according to (A.6) and (A.7) in Bercu \& Pro\"{i}a~\cite{Bercu-2013},
\begin{equation}\label{max-X-Varepsilon}
\max_{1\leq k\leq n}X_{k,n}^2\leq\frac{\kappa_n^2}{\gamma_1^2}\max_{1\leq k\leq n}\varepsilon_{k,n}^2
\quad{\rm and}\quad \max_{1\leq k\leq n}\varepsilon_{k,n}^2\leq\frac{\kappa_n^2}{\gamma_2^2}\max_{1\leq k\leq n}V_k^2,
\end{equation}
we can see that
$$
\max_{1\leq k\leq n}X_{k,n}^2\leq\frac{\kappa_n^4}{\gamma_1^2\gamma_2^2}\max_{1\leq k\leq n}V_k^2,
$$
which implies that, for any $x>0$ and the $t_0>0$ in (\ref{exp-moment}),
\begin{align*}
P\left(\max_{1\leq k\leq n}X_{k,n}^2>x\right)
&\leq P\left(\max_{1\leq k\leq n}V_k^2>\frac{x\gamma_1^2\gamma_2^2}{\kappa_n^4}\right)\\
&\leq n\exp\left\{-\frac{t_0x\gamma_1^2\gamma_2^2}{\kappa_n^4}\right\}E\exp\{t_0V_1^2\}.
\end{align*}
Hence, under (Case~I),
\begin{align*}
&\frac{1}{b_n^2}\log P\left(\frac{\max_{1\leq k\leq n}X_{k,n}^2}{b_n\sqrt{n\kappa_n^3}}>\delta\right)\\
&\leq \frac{\sqrt{n\kappa_n^{-5}}}{b_n}\left(\frac{\log n}{b_n\sqrt{n\kappa_n^{-5}}}-t_0\delta\gamma_1^2\gamma_2^2
+\frac{\log E\exp\{t_0V_1^2\}}{b_n\sqrt{n\kappa_n^{-5}}}\right)\to-\infty.
\end{align*}
On the other hand, under (Case~II),
\begin{align*}
&\frac{1}{\lambda_n^2}\log P\left(\frac{\max_{1\leq k\leq n}X_{k,n}^2}{\lambda_n\sqrt{n\kappa_n}}>\delta\right)\\
&\leq \frac{\sqrt{n\kappa_n^{-7}}}{\lambda_n}\left(\frac{\log n}{\lambda_n\sqrt{n\kappa_n^{-7}}}-t_0\delta\gamma_1^2\gamma_2^2
+\frac{\log E\exp\{t_0V_1^2\}}{\lambda_n\sqrt{n\kappa_n^{-7}}}\right)\to-\infty.
\end{align*}
Then the proof of part (a) is completed.
\vskip5pt

For part (b), by (A.14) and (A.23) in Bercu \&  Pro\"{i}a~\cite{Bercu-2013}, we know that
\begin{equation}\label{decom-S-P}
\left\{\begin{array}{ll}
S_n=\frac{1+\theta_n\rho_n}{(1-\theta_n\rho_n)(1-\theta_n^2)(1-\rho_n^2)}(L_n+R_{n1}),\\
P_n=\theta^*_nS_{n-1,n}+\frac{1}{1+\theta_n\rho_n}M_n+\frac{\theta_n\rho_n}{1+\theta_n\rho_n}X_{n,n}X_{n-1,n}.\\
\end{array}\right.
\end{equation}
where~$L_n=\sum_{k=1}^n{V_k}^2$ and $R_{n1}$ is the remainder term.
By simple calculations, we can write that
\begin{equation}\label{eq-11}
\frac{1+\theta_n\rho_n}{(1-\theta_n\rho_n)(1-\theta_n^2)(1-\rho_n^2)}
=\left\{\begin{array}{ll}
-\frac{\kappa_n^3}{2\gamma_1\gamma_2(\gamma_1+\gamma_2)}+o(\kappa_n^3),& \textrm{Case~I}\\
-\frac{\kappa_n(\gamma_1+\gamma_2)}{8\gamma_1\gamma_2}+o(\kappa_n),& \textrm{Case~II}\\
\end{array}\right.
\end{equation}
and
\begin{equation}\label{eq-11-1}
\theta_n^*=\left\{\begin{array}{ll}
1+O(\kappa_n^{-2}),& \textrm{Case~I}\\
\frac{\gamma_2-\gamma_1}{\gamma_1+\gamma_2}+O(\kappa_n^{-1}),& \textrm{Case~II}\\
\end{array}\right.,\;\;
\frac{1}{1+\theta_n\rho_n}=\left\{\begin{array}{ll}
1/2+O(\kappa_n^{-1}),& \textrm{Case~I}\\
-\frac{\kappa_n}{\gamma_1+\gamma_2}+O(1),& \textrm{Case~II}\\
\end{array}\right..
\end{equation}
According to Lemmas \ref{mdp-V}, \ref{lem-martingales-ldp} and part (a) of this lemma, we have
\begin{equation}\label{remainder-R1}
\lim_{n\to\infty}\frac{1}{a_n^2}\log P\left(\frac{|R_{n1}|}{n}>\delta\right)=-\infty
\end{equation}
and
\begin{equation}\label{remainder-R1-1}
\lim_{n\to\infty}\frac{1}{a_n^2}\log P\left(\left|\frac{L_n}{n}
-\sigma^2\right|>\delta\right)=-\infty,
\end{equation}
where $a_n=b_n$ in (Case~I) and $a_n=\lambda_n$ in (Case~II).
Then, part (b) can be achieved by (\ref{decom-S-P})-(\ref{remainder-R1-1}).
\vskip5pt

Now, we turn to the proof of part (c). Similar to the proof of part (a), we can get that
\begin{equation}\label{deviation-ineq-max-varepsilon}
\lim_{n\to\infty}\frac{1}{a_n^2}\log P\left(\frac{\max_{1\leq k\leq n}\varepsilon_{k,n}^2}{a_n\sqrt{n\kappa_n}}>\delta\right)=-\infty,
\end{equation}
where $a_n=b_n$ in (Case~I) and $a_n=\lambda_n$ in (Case~II).
Moreover, since
$$
T_n=\frac{L_n+2\rho_nU_n-\rho_n^2\varepsilon_{n,n}^2}{1-\rho_n^2}
=-\frac{\kappa_n\left(L_n+2\rho_nU_n-\rho_n^2\varepsilon_{n,n}^2\right)}{\gamma_2\left(2+\frac{\gamma_2}{\kappa_n}\right)},
$$
applying (\ref{remainder-R1-1}), (\ref{deviation-ineq-max-varepsilon}) and Lemma \ref{lem-martingales-ldp}, we can complete the proof of part (c).
\vskip5pt

Finally, if note that
\begin{equation}\label{Q-formulation}
Q_n
=\frac{1}{2}\left((1-\theta_n^2)S_n+\theta_n^2X_{n,n}^2+T_n\right),
\end{equation}
then, parts (a)-(c) in this lemma, immediately yield
part (d).

\subsection{Proof of Lemma \ref{lem-mdp-Z}}

In order to prove Lemma \ref{lem-mdp-Z}, we need a simplified version on the results of Puhalskii \cite{Puhalskii} as to the
moderate deviations for a sequence of martingale differences which is also stated as Theorem 4.9 in Bitseki Penda {\it et al.} \cite{Penda}.

\begin{lem}[Bitseki Penda {\it et al.}\cite{Penda}, Puhalskii \cite{Puhalskii}]\label{mdp-martingale-P}
Let $\{m_{k,n}, 1\leq k\leq n\}$ be a triangular array of martingale differences with values in~$\mathbb{R}^d$,
with respect to the filtration~$\{\mathcal{F}_{k,n}, 1\leq k\leq n\}$. Let $\{a_n, n\geq1\}$ be a sequence of real numbers
satisfying that as~$n\to+\infty$
$$
a_n\to\infty,\quad\frac{n}{a_n^2}\to\infty.
$$
Suppose that there exists a symmetric positive-semidefinite matrix $Q$ such that, for any $\delta>0$,
\begin{equation}\label{P-1}
\lim_{n\to\infty}\frac{1}{a_n^2}\log P\left(\left\|\frac{1}{n}\sum_{k=1}^{n}E\left(m_{k,n}m_{k,n}^{\tau}|\mathcal{F}_{k-1,n}\right)-Q\right\|>\delta\right)
=-\infty,
\end{equation}
where $\|M\|$ stands for the Euclidean norm of matrix $M$. Suppose that there exists a constant $c>0$ such that for each~$1\leq k\leq n$,
\begin{equation}\label{P-2}
|m_{k,n}|\leq c\frac{\sqrt{n}}{a_n},\quad a.s.
\end{equation}
where $|\upsilon|$ stands for the Euclidean norm of vector $\upsilon$.
Moreover, suppose that for all~$a>0$ and $\delta>0$, we have the exponential Lindeberg's condition
\begin{equation}\label{P-3}
\lim_{n\to\infty}\frac{1}{a_n^2}\log P\left(\left|\frac{1}{n}\sum_{k=1}^{n}E\left(|m_{k,n}|^2I_{\{|m_{k,n}|\geq a\frac{\sqrt{n}}{a_n}\}}\Big|\mathcal{F}_{k-1,n}\right)\right|>\delta\right)
=-\infty.
\end{equation}
Then, the sequence $\left\{\frac{\sum_{k=1}^{n}m_{k,n}}{\sqrt{n}a_n}, n\geq1\right\}$ satisfies the large deviation principle with speed $a_n^2$ and good rate function
$$
\Lambda^*(x)=\sup_{\lambda\in\mathbb{R}^d}
\left\{\lambda^{\tau}x-\frac{1}{2}\lambda^{\tau}Q\lambda\right\}, \quad x\in\mathbb{R}^d.
$$
In particular, if~$Q$ is invertible, $\Lambda^*(x)=\frac{1}{2}x^{\tau}Q^{-1}x$.
\end{lem}

\noindent{\it\textbf{Proof of Lemma \ref{lem-mdp-Z}}.}
We will give the proof of Lemma \ref{lem-mdp-Z} under (Case~I) and (Case~II) respectively.

\vskip5pt

\noindent\textbf{Case I.} Firstly, we introduce the following modifications of the martingale arrays $\left\{M_{l,n}, 1\leq l\leq n\right\}$
and $\left\{U_{l,n}, 1\leq l\leq n\right\}$, which are slightly different from the equation (4.77) in Bitseki Penda {\it et al.} \cite{Penda}. For any $r>0$ and $1\leq l\leq n$,
\begin{equation}\label{modification-MU}
M_{l,n}^{(r)}:=\sum_{k=1}^lX_{k-1,n}^{(r)}V_k^{(n)},\;U_{l,n}^{(r)}:=\sum_{k=1}^l\varepsilon_{k-1,n}^{(r)}V_k^{(n)},
\;M_{n,n}^{(r)}:=M_n^{(r)},~ U_{n,n}^{(r)}:=U_n^{(r)},
\end{equation}
where
\begin{equation}\label{modification-VX}
\begin{aligned}
&V_k^{(n)}:=V_kI_{\left\{|V_k|\leq \sqrt{\kappa_n}\right\}}-E\left(V_kI_{\left\{|V_k|\leq \sqrt{\kappa_n}\right\}}\right),\\
X_{k,n}^{(r)}:=&X_{k,n}I_{\left\{|X_{k,n}|\leq r\frac{\sqrt{n\kappa_n^2}}{b_n}\right\}}
\quad\mbox{and}\quad \varepsilon_{k,n}^{(r)}:=\varepsilon_{k,n}I_{\left\{|\varepsilon_{k,n}|\leq r\frac{\sqrt{n}}{b_n}\right\}}.
\end{aligned}
\end{equation}
Let
$$
Z_{l,n}^{(r)}=\begin{pmatrix}\frac{M_{l,n}^{(r)}}{\kappa_n}\\\\U_{l,n}^{(r)}\end{pmatrix}:=\sum_{k=1}^lm_{k,n}^{(r)},
\qquad
Z_{n}^{(r)}:=Z_{n,n}^{(r)}.
$$

We divide our proofs into the following two steps.
\vskip5pt

\textbf{Step I.} {\it To prove that $\left\{\frac{Z_n^{(r)}}{b_n\sqrt{n\kappa_n}}, n\geq1\right\}$
satisfies the large deviations with speed $b_n^2$ and good rate function
$$
I_Z(x)=\frac{x^{\tau}\Theta^{-1} x}{2}, \quad x\in\mathbb{R}^2.
$$
}

In fact, we only need to verify the conditions (\ref{P-1})-(\ref{P-3}) in Lemma \ref{mdp-martingale-P}.
Note that, for any $1\leq k\leq n$,
\begin{equation}\label{eq-33}
\frac{1}{\kappa_n}|X_{k-1,n}^{(r)}V_k^{(n)}|\leq r\frac{\sqrt{n\kappa_n}}{b_n},
\qquad |\varepsilon_{k-1,n}^{(r)}V_k^{(n)}|\leq r\frac{\sqrt{n\kappa_n}}{b_n},
\end{equation}
hence, for all $a>0$, we can write that
\begin{align*}
&\frac{1}{\sigma_n^2}\sum_{k=1}^{n}E\left(|m_{k,n}^{(r)}|^2I_{\left\{|m_{k,n}^{(r)}|>a\frac{\sqrt{n}}{b_n}\right\}}\bigg|\mathcal{F}_{k-1}\right)\\
&\leq \sum_{k=1}^{n}\left(\frac{X_{k,n}^2}{\kappa_n^2}+\varepsilon_{k,n}^2\right)
I_{\left\{|\frac{X_{k,n}}{\kappa_n}|+|\varepsilon_{k,n}|>a\frac{\sqrt{n}}{b_n}\right\}}\\
&\leq\frac{1}{\kappa_n^2}\sum_{k=1}^{n}X_{k,n}^2I_{\left\{|X_{k,n}|>a\frac{\sqrt{n\kappa_n^2}}{2b_n}\right\}}
+\sum_{k=1}^{n}\varepsilon_{k,n}^2I_{\left\{|\varepsilon_{k,n}|>a\frac{\sqrt{n}}{2b_n}\right\}}\\
&\quad+\frac{1}{\kappa_n^2}\sum_{k=1}^{n}X_{k,n}^2I_{\left\{|\varepsilon_{k,n}|>a\frac{\sqrt{n}}{2b_n}\right\}}
+\sum_{k=1}^{n}\varepsilon_{k,n}^2I_{\left\{|X_{k,n}|>a\frac{\sqrt{n\kappa_n^2}}{2b_n}\right\}},
\end{align*}
where $\sigma_n^2=E\left(V_1^{(n)}\right)^2$. Now applying (\ref{quar-coquar-ineq-M}), we get
$$
\sum_{k=1}^{n}X_{k,n}^2I_{\left\{|X_{k,n}|>a\frac{\sqrt{n\kappa_n^2}}{2b_n}\right\}}
\leq\frac{4b_n^2}{a^2n\kappa_n^2}\Xi_n\leq\frac{4b_n^2\kappa_n^6}{a^2\gamma_1^4\gamma_2^4n}\Lambda_n
$$
and
$$
\sum_{k=1}^{n}\varepsilon_{k,n}^2I_{\left\{|\varepsilon_{k,n}|>a\frac{\sqrt{n}}{2b_n}\right\}}
\leq\frac{4b_n^2}{a^2n}\sum_{k=1}^{n}\varepsilon_{k,n}^4\leq\frac{4b_n^2\kappa_n^4}{a^2\gamma_2^4n}\Lambda_n.
$$
Therefore we can obtain that, for any $\delta>0$,
{\small$$
P\left(\frac{1}{n\kappa_n^3}\sum_{k=1}^{n}X_{k,n}^2I_{\left\{|X_{k,n}|> a\frac{\sqrt{n\kappa_n^2}}{b_n}\right\}}>\delta\right)
\leq P\left(\frac{\Lambda_n-nEV_1^4}{b_n\sqrt{n}}
>\frac{\sqrt{n}}{b_n}\left(\frac{\delta a^2\gamma_1^4\gamma_2^4n}{4b_n^2\kappa_n^3}-EV_1^4\right)\right)
$$}
and
{\small
$$
P\left(\frac{1}{n\kappa_n}\sum_{k=1}^{n}\varepsilon_{k,n}^2I_{\left\{|\varepsilon_{k,n}|>a\frac{\sqrt{n\kappa_n}}{2b_n}\right\}}>\delta\right)
\leq P\left(\frac{\Lambda_n-nEV_1^4}{b_n\sqrt{n}}
>\frac{\sqrt{n}}{b_n}\left(\frac{\delta a^2\gamma_2^4n}{4b_n^2\kappa_n^3}-EV_1^4\right)\right),
$$
}
which implies, by the condition (H-I) and Lemma \ref{mdp-V}, that
\begin{equation}\label{eq-34}
\lim_{n\to\infty}\frac{1}{b_n^2}
\log P\left(\frac{1}{n\kappa_n^3}\sum_{k=1}^{n}X_{k,n}^2I_{\left\{|X_{k,n}|> a\frac{\sqrt{n\kappa_n^2}}{b_n}\right\}}>\delta\right)
=-\infty
\end{equation}
and
\begin{equation}\label{eq-35}
\lim_{n\to\infty}\frac{1}{b_n^2}
\log P\left(\frac{1}{n\kappa_n}\sum_{k=1}^{n}\varepsilon_{k,n}^2I_{\left\{|\varepsilon_{k,n}|>a\frac{\sqrt{n}}{2b_n}\right\}}>\delta\right)
=-\infty.
\end{equation}
Moreover, by H\"{o}lder inequality and (\ref{quar-coquar-ineq-M}),
$$
\sum_{k=1}^{n}X_{k,n}^2I_{\left\{|\varepsilon_{k,n}|>a\frac{\sqrt{n}}{2b_n}\right\}}
\leq\frac{4b_n^2}{a^2n}\sum_{k=1}^{n}X_{k,n}^2\varepsilon_{k,n}^2\leq\frac{4b_n^2\kappa_n^6}{a^2\gamma_1^2\gamma_2^4n}\Lambda_n
$$
and
$$
\sum_{k=1}^{n}\varepsilon_{k,n}^2I_{\left\{|X_{k,n}|>a\frac{\sqrt{n\kappa_n^2}}{2b_n}\right\}}
\leq\frac{4b_n^2}{a^2n\kappa_n^2}\sum_{k=1}^{n}X_{k,n}^2\varepsilon_{k,n}^2\leq\frac{4b_n^2\kappa_n^4}{a^2\gamma_1^2\gamma_2^4n}\Lambda_n.
$$
Therefore, similar to the proofs of (\ref{eq-34}) and (\ref{eq-35}), we can obtain that
\begin{equation}\label{eq-36}
\lim_{n\to\infty}\frac{1}{b_n^2}
\log P\left(\frac{1}{n\kappa_n^3}\sum_{k=1}^{n}X_{k,n}^2I_{\left\{|\varepsilon_{k,n}|> a\frac{\sqrt{n}}{b_n}\right\}}>\delta\right)
=-\infty
\end{equation}
and
\begin{equation}\label{eq-37}
\lim_{n\to\infty}\frac{1}{b_n^2}
\log P\left(\frac{1}{n\kappa_n}\sum_{k=1}^{n}\varepsilon_{k,n}^2I_{\left\{|X_{k,n}|>a\frac{\sqrt{n\kappa_n^2}}{2b_n}\right\}}>\delta\right)
=-\infty.
\end{equation}
Now, from (\ref{eq-34})-(\ref{eq-37}) and the fact that $\sigma_n^2\to\sigma^2$ as~$n\to\infty$, it follows that
\begin{equation}\label{eq-38}
\lim_{n\to\infty}\frac{1}{b_n^2}
\log P\left(\frac{1}{n\kappa_n}\sum_{k=1}^{n}E\left(|m_{k,n}^{(r)}|^2
I_{\left\{|m_{k,n}^{(r)}|>a\frac{\sqrt{n\kappa_n}}{b_n}\right\}}\bigg|\mathcal{F}_{k-1}\right)>\delta\right)
=-\infty.
\end{equation}
Hence, by the equations (\ref{eq-33}), (\ref{eq-38}) and Lemma \ref{mdp-martingale-P}, in order to complete the proof of {\textbf{Step I}, it is sufficient to show that
\begin{equation}\label{eq-39}
\lim_{n\to\infty}\frac{1}{b_n^2}
\log P\left(\left\|\frac{\langle Z_{\bullet,n}^{(r)}\rangle_n}{n\kappa_n}
-\Theta\right\|>\delta\right)=-\infty.
\end{equation}
Let $S_{l,n}^{(r)}=\sum_{k=1}^{l}\left(X_{k,n}^{(r)}\right)^2, \;S_{n,n}^{(r)}:=S_{n}^{(r)}$, then
$$
\langle M_{\bullet,n}^{(r)}\rangle_n=\sigma_n^2S_{n-1,n}^{(r)}.
$$
We can write that, for any $\delta>0$,
\begin{align*}
&P\left(\left|\frac{\langle M_{\bullet,n}^{(r)}\rangle_n}{n\kappa_n^3}
+\frac{\sigma^2\sigma_n^2}{2\gamma_1\gamma_2(\gamma_1+\gamma_2)}\right|>\delta\right)\\
&\leq P\left(\left|\frac{S_{n-1,n}}{n\kappa_n^3}
+\frac{\sigma^2}{2\gamma_1\gamma_2(\gamma_1+\gamma_2)}\right|>\frac{\delta}{2\sigma_n^2}\right)
+P\left(\left|\frac{S_{n-1,n}-S_{n-1,n}^{(r)}}{n\kappa_n^3}\right|>\frac{\delta}{2\sigma_n^2}\right).
\end{align*}
From Lemma \ref{lem-quar-coquar-ldp} and (\ref{eq-34}), it follows that
\begin{equation}\label{eq-40}
\lim_{n\to\infty}\frac{1}{b_n^2}
\log P\left(\left|\frac{\langle M_{\bullet,n}^{(r)}\rangle_n}{n\kappa_n^3}
+\frac{\sigma^2\sigma_n^2}{2\gamma_1\gamma_2(\gamma_1+\gamma_2)}\right|>\delta\right)=-\infty.
\end{equation}
Similarly, from Lemma \ref{lem-quar-coquar-ldp}, (\ref{eq-35}) and the fact that
$$
\langle U_{\bullet,n}^{(r)}\rangle_n=\sigma_{n}^2\sum_{k=1}^{n-1}\left(\varepsilon_{k,n}^{(r)}\right)^2,
$$
it follows that
\begin{equation}\label{eq-41}
\lim_{n\to\infty}\frac{1}{b_n^2}
\log P\left(\left|\frac{\langle U_{\bullet,n}^{(r)}\rangle_n}{n\kappa_n}
+\frac{\sigma^2\sigma_n^2}{2\gamma_2}\right|>\delta\right)=-\infty.
\end{equation}
The above discussions, together with (\ref{eq-40}) and (\ref{eq-41}), show that, in order to get (\ref{eq-39}), we only have to
prove that
\begin{equation}\label{eq-42}
\lim_{n\to\infty}\frac{1}{b_n^2}
\log P\left(\left|\frac{\langle X_{\bullet,n}^{(r)}, U_{\bullet,n}^{(r)}\rangle_n}{n\kappa_n^2}
-\frac{\sigma^2\sigma_n^2}{2\gamma_2(\gamma_1+\gamma_2)}\right|>\delta\right)=-\infty.
\end{equation}
In fact,
$$
\langle X_{\bullet,n}^{(r)}, U_{\bullet,n}^{(r)}\rangle_n=\sigma_{n}^2\sum_{k=1}^{n-1}X_{k,n}^{(r)}\varepsilon_{k,n}^{(r)},
$$
which implies that
{\small\begin{equation}\label{eq-43}
\begin{aligned}
&P\left(\left|\frac{\langle X_{\bullet,n}^{(r)}, U_{\bullet,n}^{(r)}\rangle_n}{n\kappa_n^2}
-\frac{\sigma^2\sigma_n^2}{2\gamma_2(\gamma_1+\gamma_2)}\right|>\delta\right)\\
&\leq P\left(\left|\frac{Q_{n-1,n}}{n\kappa_n^2}
-\frac{\sigma^2}{2\gamma_2(\gamma_1+\gamma_2)}\right|>\frac{\delta}{3\sigma_n^2}\right)
+P\left(\frac{\left|\sum_{k=1}^{n-1}X_{k,n}\left(\varepsilon_{k,n}^{(r)}-\varepsilon_{k,n}\right)\right|}{n\kappa_n^2}
>\frac{\delta}{3\sigma_n^2}\right)\\
&\quad+P\left(\frac{\left|\sum_{k=1}^{n-1}\varepsilon_{k,n}^{(r)}\left(X_{k,n}^{(r)}-X_{k,n}\right)\right|}{n\kappa_n^2}
>\frac{\delta}{3\sigma_n^2}\right).
\end{aligned}
\end{equation}}
Now, by H\"{o}lder inequality, we know that
{\small\begin{align*}
&P\left(\frac{\left|\sum_{k=1}^{n-1}X_{k,n}\left(\varepsilon_{k,n}^{(r)}-\varepsilon_{k,n}\right)\right|}{n\kappa_n^2}
>\frac{\delta}{3\sigma_n^2}\right)\\
&\leq P\left(\frac{S_{n-1,n}}{n\kappa_n^3}>-\frac{\sigma^2}{\gamma_1\gamma_2(\gamma_1+\gamma_2)}\right)
+P\left(\frac{1}{n\kappa_n}\sum_{k=1}^{n}\varepsilon_{k,n}^2I_{\left\{|\varepsilon_{k,n}|>r\frac{\sqrt{n}}{2b_n}\right\}}
>-\frac{\gamma_1\gamma_2(\gamma_1+\gamma_2)\delta^2}{9\sigma^2\sigma_n^2}\right).
\end{align*}}
From Lemma \ref{lem-quar-coquar-ldp} and (\ref{eq-35}), it follows that
\begin{equation}\label{eq-44}
\lim_{n\to\infty}\frac{1}{b_n^2}
\log P\left(\frac{\left|\sum_{k=1}^{n-1}X_{k,n}\left(\varepsilon_{k,n}^{(r)}-\varepsilon_{k,n}\right)\right|}{n\kappa_n^2}
>\frac{\delta}{3\sigma_n^2}\right)=-\infty.
\end{equation}
Similarly, by  Lemma~\ref{lem-quar-coquar-ldp} and ~(\ref{eq-34}), we can obtain that
\begin{equation}\label{eq-45}
\lim_{n\to\infty}\frac{1}{b_n^2}
\log P\left(\frac{\left|\sum_{k=1}^{n-1}\varepsilon_{k,n}^{(r)}\left(X_{k,n}^{(r)}-X_{k,n}\right)\right|}{n\kappa_n^2}
>\frac{\delta}{3\sigma_n^2}\right)=-\infty.
\end{equation}
Then, Lemma \ref{lem-quar-coquar-ldp}, together with  (\ref{eq-43})-(\ref{eq-45}), yields the equation (\ref{eq-42}).

\vskip5pt

\textbf{Step II.} {\it To prove the exponential equivalence between $Z_n^{(r)}$ and $Z_n$, i.e. for any $\delta>0$,
\begin{equation}\label{eq-46}
\limsup_{n\to\infty}\frac{1}{b_n^2}\log P\left(\frac{\left|Z_n-Z_n^{(r)}\right|}{b_n\sqrt{n\kappa_n}}>\delta\right)
=-\infty.
\end{equation}}
In fact,
\begin{equation}\label{eq-47}
M_n-M_n^{(r)}=F_n^{(r)}+E_n^{(r)},
\end{equation}
where
$$
F_n^{(r)}=\sum_{k=1}^{n}\left(X_{k-1,n}-X_{k-1,n}^{(r)}\right)V_k,~
E_{l,n}^{(r)}=\sum_{k=1}^{l}X_{k-1,n}^{(r)}\left(V_k-V_k^{(n)}\right),~E_n^{(r)}:=E_{n,n}^{(r)}.
$$
From (\ref{quar-coquar-ineq-M}), we have
\begin{align*}
|F_n^{(r)}|&\leq\frac{b_n^2}{r^2n\kappa_n^2}\left|\sum_{k=1}^{n}X_{k-1,n}^3V_k\right|\\
&\leq\frac{b_n^2}{r^2n\kappa_n^2}\left(\sum_{k=1}^{n}X_{k-1,n}^4\right)^{\frac{1}{2}}
\left(\sum_{k=1}^{n}X_{k-1,n}^2V_k^2\right)^{\frac{1}{2}}
\leq\frac{\kappa_n^4b_n^2}{r^2\gamma_1^3\gamma_2^3n}\Lambda_n.
\end{align*}
Hence, by condition (H-I) and Lemma \ref{mdp-V}, one can see that
\begin{equation}\label{eq-48}
\begin{aligned}
&\limsup_{n\to\infty}\frac{1}{b_n^2}\log P\left(\frac{|F_n^{(r)}|}{b_n\sqrt{n\kappa_n^3}}>\delta\right)\\
&\leq
\limsup_{n\to\infty}\frac{1}{b_n^2}\log P\left(\frac{\Lambda_n-nEV_1^4}{\sqrt{n}b_n}>
\frac{\sqrt{n}}{b_n}\left(\frac{\delta r^2\gamma_1^3\gamma_2^3\sqrt{n}}{\sqrt{\kappa_n^5}b_n}-EV_1^4\right)\right)=-\infty.
\end{aligned}
\end{equation}

As for the item $E_n^{(r)}$, we also claim that
\begin{equation}\label{eq-49}
\limsup_{n\to\infty}\frac{1}{b_n^2}\log P\left(\frac{\left|E_n^{(r)}\right|}{b_n\sqrt{n\kappa_n^3}}>\delta\right)
=-\infty.
\end{equation}
In fact, $\left\{E_{l,n}^{(r)}, 1\leq l\leq n\right\}$ is a locally square-integrable martingale array.
For all $1\leq k\leq n$, we have
\begin{align*}
&P\left(\left|X_{k-1,n}^{(r)}\left(V_k-V_k^{(n)}\right)\right|>b_n\sqrt{n\kappa_n^3}\right)\\
&\leq P\left(\left|V_1-V_1^{(n)}\right|>b_n^2\sqrt{\kappa_n}\right)\\
&\leq\exp\left\{-\frac{1}{2}t_0b_n^2\sqrt{\kappa_n}\right\}E\exp\left\{t_0|V_1|\right\},
\end{align*}
which implies that
\begin{equation}\label{eq-50}
\limsup_{n\to\infty}\frac{1}{b_n^2}\log\left(n \esssup_{1\leq k\leq n}P\left(\left|X_{k-1,n}^{(r)}\left(V_k-V_k^{(n)}\right)\right|>b_n\sqrt{n\kappa_n^3}\,\big|\mathcal{F}_{k-1}\right)\right)=-\infty.
\end{equation}
Moreover, for all $a>0$ and $\delta>0$,
\begin{align*}
 &P\left(\frac{1}{n\kappa_n^3}\sum_{k=1}^{n}\left(E\left(X_{k-1,n}^{(r)}(V_k-V_k^{(n)})\right)^2
I_{\left\{|X_{k-1,n}^{(r)}(V_k-V_k^{(n)})|> a\frac{\sqrt{n\kappa_n^3}}{b_n}\right\}}\bigg|\mathcal{F}_{k-1}\right)>\delta\right)\\
&\leq
P\left(\frac{1}{n\kappa_n^3}\sum_{k=1}^{n}\left(X_{k-1,n}^{(r)}\right)^2
I_{\left\{|X_{k-1,n}^{(r)}|> a\frac{\sqrt{n\kappa_n^2}}{b_n^3}\right\}}>\frac{\delta}{\tilde{\sigma}_n^2}\right)
+P\left(\max_{1\leq k\leq n}|V_k-V_k^{(n)}|\geq b_n^2\sqrt{\kappa_n}\right)\\
&\leq P\left(\frac{\Lambda_n-nEV_1^4}{b_n\sqrt{n}}
>\frac{\sqrt{n}}{b_n}\left(\frac{\delta a^2\gamma_1^4\gamma_2^4n}{b_n^6\kappa_n^3\tilde{\sigma}_n^2}-EV_1^4\right)\right)
+nP\left(|V_1-V_1^{(n)}|\geq b_n^2\sqrt{\kappa_n}\right),
\end{align*}
where $\tilde{\sigma}_n^2:=E(V_1-V_1^{(n)})^2$.
Now, since $\tilde{\sigma}_n^2\leq \frac{EV_1^4}{\kappa_n}$, we obtain that, by Lemma \ref{mdp-V},
\begin{align*}
&\limsup_{n\to\infty}\frac{1}{b_n^2}\log P\left(\frac{\Lambda_n-nEV_1^4}{b_n\sqrt{n}}
>\frac{\sqrt{n}}{b_n}\left(\frac{\delta a^2\gamma_1^4\gamma_2^4n}{b_n^6\kappa_n^3\tilde{\sigma}_n^2}-EV_1^4\right)\right)\\
&\leq \limsup_{n\to\infty}\frac{1}{b_n^2}\log P\left(\frac{\Lambda_n-nEV_1^4}{b_n\sqrt{n}}
>\frac{\sqrt{n}}{b_n}\left(\frac{\delta a^2\gamma_1^4\gamma_2^4n}{b_n^6\kappa_n^2EV_1^4}-EV_1^4\right)\right)
=-\infty.
\end{align*}
Hence,
\begin{equation}\label{eq-50-1}
\begin{aligned}
&\limsup_{n\to\infty}\frac{1}{b_n^2}\log P\left(\frac{1}{n\kappa_n^3}\sum_{k=1}^{n}\left(E\left(X_{k-1,n}^{(r)}(V_k-V_k^{(n)})\right)^2
I_{\left\{|X_{k-1,n}^{(r)}(V_k-V_k^{(n)})|> a\frac{\sqrt{n\kappa_n^3}}{b_n}\right\}}\bigg|\mathcal{F}_{k-1}\right)>\delta\right)\\
&=-\infty.
\end{aligned}
\end{equation}
Furthermore, for all $\delta>0$, using Lemma \ref{lem-quar-coquar-ldp}, (\ref{eq-34}) and the fact that $\tilde{\sigma}_n^2\to0$ as $n\to\infty$, we get
\begin{equation}\label{eq-50-2}
\begin{aligned}
&\limsup_{n\to\infty}\frac{1}{b_n^2}\log P\left(\left|\frac{\langle E_{\bullet,n}^{(r)}\rangle_n}{n\kappa_n^3}\right|>\delta\right)\\
&\leq\limsup_{n\to\infty}\frac{1}{b_n^2}\log\left(P\left(\left|\frac{S_{n-1,n}}{n\kappa_n^3}\right|>\frac{\delta}{2\tilde{\sigma}_n^2}\right)
+P\left(\left|\frac{S_{n-1,n}-S_{n-1,n}^{(r)}}{n\kappa_n^3}\right|>\frac{\delta}{2\tilde{\sigma}_n^2}\right)\right)=-\infty.
\end{aligned}
\end{equation}
Finally, the above discussions, together with (\ref{eq-50})-(\ref{eq-50-2}) and Theorem 1 in Djellout \cite{Djellout}, show that $\left\{\frac{1}{b_n\sqrt{n\kappa_n^3}}E_{n}^{(r)}, n\geq1\right\}$ satisfies the large deviations with speed $b_n^2$
and good rate function
$$
I_E(x)=\left\{\begin{array}{ll}
0,& \textrm{if } x=0,\\
+\infty,& \textrm{otherwise}.\\\end{array}\right.
$$
Hence, we complete the proof of (\ref{eq-49}). Obviously, (\ref{eq-48}) and (\ref{eq-49}) imply that
\begin{equation}\label{eq-51}
\limsup_{n\to\infty}\frac{1}{b_n^2}\log P\left(\frac{\left|M_n-M_n^{(r)}\right|}{b_n\sqrt{n\kappa_n^3}}>\delta\right)
=-\infty.
\end{equation}
Similar to the proof of (\ref{eq-51}), we can obtain that
\begin{equation}\label{eq-52}
\limsup_{n\to\infty}\frac{1}{b_n^2}\log P\left(\frac{\left|U_n-U_n^{(r)}\right|}{b_n\sqrt{n\kappa_n}}>\delta\right)
=-\infty,
\end{equation}
which, together with (\ref{eq-51}), achieves the proof of (\ref{eq-46}).

\vskip5pt
\noindent{\textbf{Case~II.}}  In this case, let $V_k^{(n)}, \varepsilon_{k,n}^{(r)}, U_{l,n}^{(r)}$ be defined as in
(\ref{modification-MU}) and (\ref{modification-VX}). However,
we need a new modifications of the martingale array $\left\{M_{l,n}, 1\leq l\leq n\right\}$: for any $r>0$ and $1\leq l\leq n$,
$$
\tilde{M}_{l,n}^{(r)}:=\sum_{k=1}^l\tilde{X}_{k-1,n}^{(r)}V_k^{(n)},
\qquad\tilde{M}_{n,n}^{(r)}:=\tilde{M}_n^{(r)},
$$
where
$$
\tilde{X}_{k,n}^{(r)}:=X_{k,n}I_{\left\{|X_{k,n}|\leq r\frac{\sqrt{n\kappa_n^2}}{\lambda_n}\right\}}.
$$
Moreover, let
$$
\tilde{Z}_{l,n}^{(r)}=\begin{pmatrix}\tilde{M}_{l,n}^{(r)}\\\\U_{l,n}^{(r)}\end{pmatrix}:=\sum_{k=1}^l\tilde{m}_{k,n}^{(r)},\quad
\tilde{Z}_{n}^{(r)}:=\tilde{Z}_{n,n}^{(r)}.
$$
Similar to (Case~I), we can show that,
\begin{itemize}
\item $\left\{\frac{\tilde{Z}_n^{(r)}}{\lambda_n\sqrt{n\kappa_n}}, n\geq1\right\}$
satisfies the large deviations with speed $\lambda_n^2$ and good rate function
$$
J_{\tilde Z}(x)=\frac{x^{\tau}\tilde{\Theta}^{-1} x}{2}, \quad x\in\mathbb{R}^2;
$$

\item the exponential equivalence between $\tilde{Z}_n^{(r)}$ and $\tilde{Z}_n$, i.e. for any $\delta>0$,
$$
\limsup_{n\to\infty}\frac{1}{\lambda_n^2}\log P\left(\frac{\left|\tilde{Z}_n-\tilde{Z}_n^{(r)}\right|}{\lambda_n\sqrt{n\kappa_n}}>\delta\right)
=-\infty.
$$
\end{itemize}

The proof of Lemma \ref{lem-mdp-Z} is complete.



\section*{Acknowledgements}
{\small
The authors would like to express their great gratitude to the anonymous referee and AE for the careful reading and
insightful comments, which surely lead to an improved presentation of this paper. The authors also appreciate Huarui He for her help on simulations. The work of H. Jiang was partially supported by the Natural Science Foundation of Jiangsu Province of China (No. BK20231435) and Fundamental Research Funds for the Central Universities
(No. NS2022069), and the work of G. Yang was partially supported by the Foundation of Young Scholar of the Educational Department of Henan Province (No. 2019GGJS012). 
}

\end{document}